\documentclass{article}
\usepackage{psfrag}
\usepackage[parfill]{parskip}
\usepackage{float, graphicx}
\usepackage[]{epsfig}
\usepackage{amsmath, amsthm, amssymb,}
\usepackage{epsfig}
\usepackage{verbatim}
\usepackage{multicol}
\usepackage{url}
\usepackage{latexsym}
\usepackage{mathrsfs}
\usepackage[colorlinks, bookmarks=true]{hyperref}
\usepackage{graphicx}
\usepackage{amsmath}
\usepackage{enumerate}
\usepackage[normalem]{ulem}
\usepackage{bm}
\usepackage{dsfont}
\usepackage{stmaryrd}
\usepackage{cite}

\usepackage{color}

\makeatletter

\def\@settitle{\begin{center}%
    \bfseries
 \normalfont\LARGE\@title
  \end{center}%
}
\def\@setauthors{\begin{center}%
 \normalsize\@author
  \end{center}%
}

\makeatother

\numberwithin{equation}{section}

\renewcommand{\cal}{\mathcal}

\newcommand{\cC}{{\cal C}}
\newcommand{\cE}{{\cal E}}

\newcommand{\fa}{{\mathfrak a}}
\newcommand{\fb}{{\mathfrak b}}
\newcommand{\fc}{{\mathfrak c}}
\newcommand{\fd}{{\mathfrak d}}

\newcommand{\fK}{{\frak K}}

\newcommand{\rd}{{\rm d}}

\newcommand{\ri}{\mathrm{i}}

\newcommand{\bC}{{\mathbb C}}
\newcommand{\bE}{\mathbb{E}}

\newcommand{\bP}{\mathbb{P}}

\newcommand{\bR}{{\mathbb R}}

\newcommand{\bZ}{\mathbb{Z}}

\newcommand{\la}{\lambda}

\DeclareMathOperator{\supp}{supp}
\DeclareMathOperator{\dist}{dist}

\DeclareMathOperator{\dom}{\mathcal{D}}

\DeclareMathOperator{\OO}{O}

\DeclareMathOperator{\argmax}{argmax}
\DeclareMathOperator{\argmin}{argmin}

\renewcommand{\Re}{\mathop{\mathrm{Re}}}
\renewcommand{\Im}{\mathop{\mathrm{Im}}}

\newcommand{\deq}{\mathrel{\mathop:}=} 
 
\renewcommand{\leq}{\leqslant}
\renewcommand{\geq}{\geqslant}

\newcommand{\td}{\tilde}
\newcommand{\del}{\partial}

\newcommand{\rn}[1]{
       \romannumeral#1
}

\newcommand{\qq}[1]{[\![{#1}]\!]}

\newcommand{\beq}{\begin{equation}}
\newcommand{\eeq}{\end{equation}}

\theoremstyle{plain} 
\newtheorem{theorem}{Theorem}[section]
\newtheorem*{theorem*}{Theorem}
\newtheorem{lemma}[theorem]{Lemma}
\newtheorem*{lemma*}{Lemma}
\newtheorem{corollary}[theorem]{Corollary}
\newtheorem*{corollary*}{Corollary}
\newtheorem{proposition}[theorem]{Proposition}
\newtheorem*{proposition*}{Proposition}
\newtheorem{assumption}[theorem]{Assumption}
\newtheorem*{assumption*}{Assumption}

\newtheorem*{definition*}{Definition}

\newtheorem*{example*}{Example}
\newtheorem{remark}[theorem]{Remark}

\newtheorem*{remark*}{Remark}
\newtheorem*{remarks*}{Remarks}

\def\author#1{\par
    {\centering{\authorfont#1}\par\vspace*{0.05in}}
}

\def\titlefont{\fontsize{13}{15}\bfseries\boldmath\selectfont\centering{}}
\def\authorfont{\fontsize{13}{15}}

\let\affiliationfont\rhfont

\def\address#1{\par
    {\centering{\affiliationfont#1\par}}\par\vspace*{11pt}
}

\def\body{
\setcounter{footnote}{0}
\def\thefootnote{\alph{footnote}}
\def\@makefnmark{{$^{\rm \@thefnmark}$}}
}

\def\title#1{
    \thispagestyle{plain}
    \vspace*{-14pt}
    \vskip 79pt
    {\centering{\titlefont #1\par}}%
    \vskip 1em
}

\def\d{\mathrm{d}}

\def\epsilon{\varepsilon}
\def\i{\mathrm{i}}
\renewcommand\leq\le
\renewcommand\geq\ge

\begin{document}

\title{Rigidity and a mesoscopic central limit theorem for Dyson Brownian Motion for general $\beta$}

\vspace{1.2cm}

\noindent \begin{minipage}[c]{0.5\textwidth}
 \author{Jiaoyang Huang}
\address{Harvard University\\
   E-mail: jiaoyang@math.harvard.edu}
 \end{minipage}
 \begin{minipage}[c]{0.5\textwidth}
 \author{Benjamin Landon}
\address{Harvard University\\
   E-mail: landon@math.harvard.edu}
 \end{minipage}
 %
 %
 %
%
%


~\vspace{0.3cm}

\let\thefootnote\relax\footnote{\noindent The work of B.L. is partially supported by NSERC. } 

\begin{abstract}
We study Dyson Brownian motion with general potential $V$ and for general $\beta \geq 1$.  
For short times $t = o (1)$ and under suitable conditions on $V$ we obtain a local law and corresponding rigidity estimates on the particle locations; that is, with overwhelming probability, the particles are close to their classical locations with an almost-optimal error estimate.
Under the condition that the density of states of  the initial data is bounded below and above down to the scale $\eta_* \ll t \ll 1$, we prove a mesoscopic central limit theorem for linear statistics at all scales $N^{-1}\ll\eta\ll t$.
\end{abstract}

%



\section{Introduction}


In 1962, Dyson interpreted the $N\times N$ Gaussian ensemble (real, complex or quaternion) as the dynamical limit of matrix-valued Brownian motion $H(t)$, and observed that the eigenvalues of $H(t)$ form an interacting $N$-particle system with a logarithmic Coulomb interaction and quadratic potential.  That is, the eigenvalue process $\{ \lambda_i (t)\}_{1\leq i\leq N} $ satisfies the following system of stochastic differential equations with quadratic $V = x^2/2$ and classical $\beta =1, 2$ or $4$ (depending on the symmetry class of the Gaussian ensemble)
\beq \label{DBM}
\rd \la_i(t) = \sqrt{\frac{2}{\beta N}} \rd B_i(t) +\frac{1}{N}\sum_{j:j\neq i}\frac{\rd t}{\lambda_i(t)-\lambda_j(t)}-\frac{1}{2}V'(\lambda_i(t))\rd t,\quad i=1,2,\cdots, N,
\eeq
where $(B_1, \cdots, B_N)$ is an $N$-dimensional Brownian motion defined on a probability space with a filtration $\mathscr F=\{\mathscr F_t, t\geq 0\}$.  The initial data  ${\bm \la}(0)=(\la_1(0),\la_2(0)\cdots, \la_N(0))\in \overline{\Delta_N}$ is given by the eigenvalues of $H(0)$.  Here, $\Delta_N$ denotes the Weyl chamber
\beq
 \Delta_N=\{\{x_i\}_{1\leq i\leq N}\in \bR^N: x_1<x_2<\cdots <x_{N}\}.
\eeq
The process $\bm\la(t)=(\lambda_1(t),\lambda_2(t),\cdots, \lambda_N(t))$ defined by the stochastic differential equation system \eqref{DBM} is called the $\beta$-Dyson Brownian motion ($\beta$-DBM) with potential $V$, which is an interacting particle system with Hamiltonian of the form 
\beq
H(x_1,\cdots, x_N)\deq -\frac{1}{2N}\sum_{1\leq i\neq j\leq N}\log |x_i-x_j|+\frac{1}{2}\sum_{i=1}^{N}V(x_i).
\eeq
For the special case $\beta=2$ and $V=x^2/2$, at each fixed time $t$, the particles $\bm\la(t)$ have the same distribution as the eigenvalues of 
\beq
H(t)\stackrel{d}{=}e^{-t/2}H(0)+\sqrt{1-e^{-t}}G,
\eeq 
where $G$ is a matrix drawn from the Gaussian Unitary Ensemble (GUE).  The global eigenvalue density of the GUE follows Wigner's semi-circle distribution \cite{MR0083848}, and the local eigenvalue statistics are given by the Sine kernel\cite{MR0143556,MR0143557,MR0143558}.  Clearly, $H(t)\rightarrow G$ as $t\rightarrow \infty$ for any choice of the initial data $H(0)$, and so the system reaches a global equilibrium for $t \gg 1$.  One can also investigate the time to local equilibrium - that is, how long it takes for the local statistics to coincide with the GUE.  Dyson conjectured  \cite{MR0148397} that the time to local equilibrium should be much faster than the order $1$ global scale.   It is expected that in the bulk, an eigenvalue statistic on the scale $\eta$ should coincide with the GUE as long as $ t \gg \eta$.  To be more precise, one expects the convergence of the following three types of statistics on three types of scales.
\begin{enumerate}
\item On the \emph{macroscopic} scale, the global eigenvalue density should converge to Wigner's semi-circle distribution provided $t\gg 1$. 
\item The linear eigenvalue statistics of test functions on the \emph{mesoscopic} scale $N^{-1}\ll \eta\ll1$ should coincide with the GUE as long as $t \gg \eta$.
\item On the \emph{microscopic} scale $\OO(N^{-1})$, the local eigenvalue statistics should be given by the sine kernel as long as $t\gg N^{-1}$. 
\end{enumerate}

For the macroscopic scale, it was proven by Li, Li and Xie \cite{GDBM1, GDBM2}, that under mild conditions on $V$, the global eigenvalue density converges to a $V$-dependent equilibrium measure (which may not be the semicircle distribution for non-quadratic $V$) provided $t\gg 1$. We refer to  \cite{MR2760897} for a nice presentation on the dynamical approach to Wigner's semi-circle law.

The time to equilibrium at the microscopic scale was studied in a series of works \cite{MR2810797,MR3372074,MR2661171,MR2662426,MR3098073,MR2481753, MR2537522,MR2871147,MR2981427,MR2639734},  by Erd{\H{o}}s, Yau and their collaborators.  For classical $\beta =1, 2, 4$, quadratic $V$ and initial data a Wigner matrix, it was proven that after a short time $t \gg N^{-1}$ the local statistics coincide with the G$\beta$E.  Later, the works \cite{ly,kevin3} established single gap universality for classical DBM for a broad class of initial data, relying on the discrete di-Giorgi-Nash-Moser theorem developed in \cite{MR3372074}.  Fixed energy universality was established in \cite{MR3541852,fix2} by developing a sophisticated homogenization theory for discrete parabolic systems.  These results are a crucial component in proving bulk universality for various classes of random matrix ensembles.  Another approach to universality, applicable in special cases was developed independently and in parallel by Tao and Vu \cite{tv1}.

A central and basic tool in the study of the local statistics of random matrices is the local law and the associated rigidity estimates.  The local law is usually formulated in terms of concentration of the Stieltjes transform of the empirical eigenvalue density at short scales $\eta \gtrsim N^{-1}$.  Rigidity estimates give high probability concentration estimates for the eigenvalue locations.  These results were first established for Wigner matrices in a series of papers \cite{MR2481753, MR2537522,MR2871147,MR2981427,MR3109424}, then extended to other matrix models, i.e. sparse random matrices \cite{MR3098073}, deformed Wigner ensembles \cite{MR3502606, ly}. Beyond matrix models, rigidity estimates have been established for one-cut and multi-cut $\beta$-ensembles \cite{MR3192527, MR3253704,MR2905803, multicutRig}, and two-dimensional Couloub gas \cite{Roland2015, Leble2015}.

For the special case of classical $\beta =1, 2, 4$ and quadratic potential $V$, the solution of \eqref{DBM} is given by a matrix model and so the methods developed for deformed Wigner matrices \cite{MR3502606,kevin1,ly} yield a local law for the Stieltjes transform of empirical eigenvalue density,
\beq
\td m_t(z)\deq \frac{1}{N} \sum_{i=1}^{N} \frac{1}{\la_i(t)-z}.
\eeq
However for nonclassical $\beta$ or non-quadratic $V$, the process \eqref{DBM} is not given by a matrix model and so a corresponding local law is not known.

Our first main result is to establish a local law for the Stieltjes transform $\td m_t (z)$ for short scales and all short times $t \ll 1$.  This result is stated as Theorem \ref{t:rigidity} below.  This implies a rigidity estimate for the particle locations $\lambda_i(t)$, i.e., that they are close to deterministic classical locations with high probability.  

Our methods are purely dynamical and do not rely on any matrix representation.  Instead, our method is based on analyzing the stochastic differential equation of the Stieltjes transform $\td m_t$ along the characteristics of the limiting continuum equation.  We remark that since the $\beta$-ensemble is the equilibrium measure of $\beta$-DBM, our results may be used to provide another proof for the rigidity of $\beta$-ensemble in the case $\beta\geq 1$, provided that one takes some large deviation estimates (such as \cite{deviation}) as input.  We also comment that the method of characteristics has recently been used, independently and in parallel, for the analysis of a  different equation in \cite{extremegap}. 

Relying on our local law we then prove a mesoscopic central limit theorem for linear statistics of the particle process on scales $\eta \ll t$.  This is stated as Theorem \ref{t:mesoCLT} below.  In particular we see that equilibrium holds for the process \eqref{DBM} on mesoscopic scales $\eta \ll t$.  Central limit theorems for mesoscopic linear statistics of Wigner matrices at all scales were established in a series of papers \cite{MR1678012,MR1689027, mesoCLT1,mesoCLT3}. Analogous results for invariant ensembles were proved in \cite{mesoCLT4,mesoCLT2}. Mesoscopic statistics for DBM with $\beta=2$ and quadratic potential was established in \cite{mesoCLTDBM}. It was proven that at mesoscopic scale $\eta$, the mesoscopic central limit theorem holds if and only if $t\gg \eta$. Recently, related results were proven for classical $\beta$ and the quadratic potential in \cite{fix2}. The analysis in \cite{mesoCLTDBM} relied on the Br\'ezin-Hikami formula special to the $\beta=2$ case, and the analysis in \cite{fix2} relied on the matrix model which exists only for classical $\beta$, i.e. $\beta=1,2,4$, neither of which are applicable here. Our approach is based on a direct analysis of the stochastic differential equation of $\td m_t$, where the leading fluctuation term is an integral with respect to Brownian motions.  The central limit theorem follows naturally for all $\beta\geq 1$ and general potential $V$.  

Finally we remark that by combining the rigidity results proven here and the methodology of \cite{fix2} one can prove gap universality for the process \eqref{DBM}, thus yielding equilibrium on the local scale $\eta =1/ N$.

We now outline the organization of the rest of the paper.   In Section \ref{s:background}, we collect some properties of $\beta$-DBM \eqref{DBM}, i.e., the existence and uniqueness of strong solutions and the existence and uniqueness of the hydrodynamic limit of the empirical density $\td \mu_t$, which is a measure valued process $\mu_t$. For quadratic $V$, these statements were proved by Chan \cite{MR1176727} and Rogers and Shi \cite{MR1217451}. For general potentials (under Assumption \ref{a:asumpV} below), the $\beta$-DBM was studied by Li, Li and Xie \cite{GDBM1, GDBM2}. 
In the second part of Section \ref{s:background}, we study the Stieltjes transform of the limit measure valued process $\mu_t$ by the method of characteristics, which are used throughout the rest of the paper.

Section \ref{s:rigidity} contains the main novelty of this paper, in which we prove the local law and rigidity estimate of the particles Theorem \ref{t:rigidity}.  We directly analyze the stochastic differential equation satisfied by $\td m_t$ using the method of characteristics.   In Section \ref{s:mesoCLT}, we prove that the linear statistics satisfy a central limit theorem at  mesoscopic scales.

In the rest of this paper, we use $C$ to represent large universal constant, and $c$ a small universal constant, which may depend on other universal constants, i.e., the constants $\fa, \fb, \fK$ in Assumptions \ref{a:asumpV} and \ref{a:asumpL}, and may be different from line by line. We write that $X=O(Y)$ if there exists some universal constant such that $|X|\leq CY$. We write $X=o(Y)$, or $X\ll Y$ if the ratio $|X|/Y\rightarrow 0$ as $N$ goes to infinity. We write $X\asymp Y$ if there exist universal constants such that $cY\leq |X|\leq CY$. We denote the set $\{1, 2,\cdots, N\}$ by $\qq{1,N}$. We say an event $\Omega$ holds with overwhelming probability, if for any $D>0$, and $N\geq N_0(D)$ large enough, $\bP(\Omega)\geq 1-N^{-D}$.

\noindent \textbf{Acknolwedgements:} We thank Paul Bourgade, Philippe Sosoe and Horng-Tzer Yau for helpful discussions and useful comments on our preliminary draft.

\section{Background on $\beta$-Dyson Brownian Motion }\label{s:background}
In this section we collect several properties of $\beta$-DBM, required in the remainder of the paper. More precisely, we state the existence and uniqueness of the strong solution to \eqref{DBM} and a weak convergence result for the empirical particle density.

In the rest of the paper, we make the following assumption on the potential $V$.
\begin{assumption}\label{a:asumpV}
We assume that the potential $V$ is a $C^4$ function, and that there exists a constant $\fK\geq 0$ such that $\inf_{x\in \bR} V''(x)\geq -2\fK$. 
\end{assumption}

We denote $M_1(\bR)$ the space of probability measures on $\bR$ and equip this space with the weak topology.  For $T>0$ we denote by $C([0,T], M_1(\bR))$ the space of continuous processes on $[0,T]$ taking values in $M_1(\bR)$.  We have the following existence result from \cite{GDBM1}.

\begin{theorem}
Suppose that $V$ satisfies Assumption \ref{a:asumpV}. For all $\beta\geq 1$ and initial data $\bm\la(0)\in \overline{\Delta_N}$, there exists a strong solution $(\bm \la(t))_{t\geq0}\in C(\bR_+,\overline{\Delta_N})$ to the stochastic differential equation \eqref{DBM}. For any $t>0$, $\bm\la(t)\in \Delta_N$ and $\bm \la(t)$ is a continuous function of $\bm\la(0)$.
\end{theorem}
\begin{proof}
The existence of strong solution with initial data $\bm\la(0)\in \Delta_N$ follows from \cite[Theorem 1.2]{GDBM1}. Following the same argument in \cite[Proposition 4.3.5]{MR2760897}, we can extend the statement to $\bm\la(0)\in \overline{\Delta_N}$ by the following comparison lemma (the special case with potential $V\equiv0$ is proved in \cite[Lemma 4.3.6]{MR2760897} and the proof below is based on the proof given there) between strong solutions of \eqref{DBM} with initial data in $\Delta_N$. 
\end{proof}

\begin{lemma}
Suppose that $V$ satisfies the Assumption \ref{a:asumpV}. Let $(\bm\la(t))_{t\geq0}$ and $(\bm \eta(t))_{t\geq 0}$ be two strong solutions of \eqref{DBM} with initial data $\bm\la(0)\in \Delta_N$ and $\bm\eta(0)\in \Delta_N$. Assume that $\la_i(0)>\eta_i(0)$ for all $i\in\qq{1,N}$. Then, almost surely, for all $t\geq 0$ and $i\in\qq{1,N}$,
\beq
 0\leq \la_i(t)-\eta_i(t)\leq e^{{\fK}t}\max_{j\in\qq{1,N}}\{\la_j(0)-\eta_j(0)\}.
\eeq
\end{lemma}
\begin{proof}
By taking difference of the stochastic differential equations satisfied by $(\bm \la(t))_{t\geq 0}$ and $(\bm \eta(t))_{t\geq 0}$, we have
\beq\label{e:diffDBM1}
\del_t(\la_i(t)-\eta_i(t))=\frac{1}{N}\sum_{j:j\neq i}\frac{(\la_j(t)-\eta_j(t))-(\la_i(t)-\eta_i(t))}{(\la_i(t)-\la_j(t))(\eta_i(t)-\eta_j(t))}\rd t -\frac{1}{2}\left(V'(\la_i(t))-V'(\eta_i(t))\right) \rd t.
\eeq

Let $i_0=\argmax_{i\in\qq{N}}\{\la_i(t)-\eta_i(t)\}$.  For $i=i_0$, the first term of \eqref{e:diffDBM1} is non-positive, and 
\beq
\del_t(\la_{i_0}(t)-\eta_{i_0}(t))\leq -\frac{1}{2}\left(V'(\la_{i_0}(t))-V'(\eta_{i_0}(t))\right).
\eeq
Either $\la_{i_0}(t)-\eta_{i_0}(t)< 0$, or using  Assumption \ref{a:asumpV} the above equation implies
$\del_t(\la_{i_0}(t)-\eta_{i_0}(t))\leq {\fK}(\la_{i_0}(t)-\eta_{i_0}(t))$. 
Hence,
\beq
\del_t ( \la_{i_0 }  (t) - \eta_{i_0} (t) )_+ \leq \fK ( \la_{i_0} (t) - \eta_{i_0} (t) )_+.
\eeq
Therefore, it follows from Gronwall's inequality,
\beq
\max_{i\in\qq{N}}\{\la_i(t)-\eta_i(t)\}\leq e^{{\fK}t}\max_{i\in\qq{N}}\{\la_i(0)-\eta_i(0)\}.
\eeq
Similarly, let $i_0=\argmin_{i\in\qq{N}}\{\la_i(t)-\eta_i(t)\}$.  Either $\la_{i_0}(t)-\eta_{i_0}(t)> 0$, or
$\del_t(\la_{i_0}(t)-\eta_{i_0}(t))\geq {\fK}(\la_{i_0}(t)-\eta_{i_0}(t))$. Again by Gronwall's inequality we obtain that $\min_{i\in\qq{N}}\{\la_i(t)-\eta_i(t)\}\geq 0$.
\end{proof}

The following theorem is a consequence of  \cite[Theorem 1.1 and 1.3]{GDBM1}.  It establishes the existence of a solution to the limiting hydrodynamic equation of the empirical particle process.  In its statement we distinguish the parameter $L$ from $N$.  This is due to the fact that we will compare the empirical measure  $\td \mu_t$ to a solution of the equation \eqref{e:MVeq} with initial data coming from the initial value of $\bm \la (0)$ which is a finite $N$ object.  The existence of this solution is  easily established using the theorem below by introducing  an auxilliary process $\bm \la^{(L)}$ which converges to $\td \mu_0$ (a fixed finite $N$ object) as $L \to \infty$.
\begin{theorem}\label{thm:limitm}
Suppose $V$ satisfies the Assumption \ref{a:asumpV}.  Let $\beta\geq 1$.  Let  $\bm \la^{(L)}(0)=(\la_1^{(L)}(0),\la_2^{(L)}(0)\cdots, \la_L^{(L)}(0))\in \overline{\Delta_N}$ be a sequence of initial data satisfying
\beq
\sup_{L>0}\frac{1}{L}\sum_{i=1}^L\log(\la_i^{(L)}(0)^2+1)<\infty.
\eeq 
Assume that the empirical measure
$\td\mu^{(L)}_0=\frac{1}{L}\sum_{i=1}^L\delta_{\la_i^{(L)}(0)}$ converges weakly as $L$ goes to infinity to $\mu_0\in M_1(\bR)$.

Let ${\bm \la}^{(L)}(t)=(\la^{(L)}_1(t),\cdots, \la^{(L)}_L(t))_{t\geq 0}$ be the solution of \eqref{DBM} with initial data $\bm \la^{(L)}(0)$, and set 
\beq
\td\mu_t^{(L)}=\frac{1}{L}\sum_{i=1}^L\delta_{\la_i^{(L)}(t)}.
\eeq
Then for any fixed time $T$, $(\td \mu_t^{(L)})_{t\in[0,T]}$ converges almost surely in $C([0,T], M_1(\bR))$. Its limit is the unique measure-valued process $(\mu_t)_{t\in [0,T]}$ characterized by the  McKean-Vlasov equation, i.e., for all $f\in C_b^2(\bR)$, $t\in [0,T]$, 
\beq\label{e:MVeq}
\del_t \int_\bR f(x)\rd \mu_t(x) =\frac{1}{2}\int\int_{\bR^2}\frac{\del_x f(x)-\del_yf(y)}{x-y}\rd \mu_t(x)\rd \mu_t(y)-\frac{1}{2}\int_\bR V'(x)f'(x)\rd \mu_t(x).
\eeq
\end{theorem}

Taking $f(x)=(x-z)^{-1}$ for $z\in \bC\setminus \bR$ in \eqref{e:MVeq}, we see that the Stieltjes transform of the limiting measure-valued process, which is defined by
\beq
m_t(z)=\int (x-z)^{-1}\rd\mu_t(x) ,
\eeq
satisfies the equation
\beq \label{e:limitm0}
\del_t m_t(z)=m_t(z)\del_z m_t(z)+\frac{1}{2}\int_{\bR}\frac{V'(x)}{(x-z)^2}\rd \mu_t(x).
\eeq

In a moment we will introduce a spatial cut-off of $V$.  In order to do this we require the following exponential bound for $ || \lambda_i (t) ||_\infty$.

\begin{proposition}\label{normbound}
Suppose $V$ satisfies Assumption \ref{a:asumpV}. Let $\beta\geq 1$, and $\bm \la(0)\in \overline{\Delta_N}$. Let $\fa$ be a constant  such that the initial data $\|\bm \la(0)\|_\infty\leq \fa$.  Then for any fixed time $T$, there exists a finite constant $\fb=\fb(\fa, T )$, such that for any $0\leq t\leq T$, the unique strong solution of \eqref{DBM} satisfies:
\beq\label{e:normbound}
\bP(\max\{|\la_1(t)|,|\la_N(t)|\}\geq \fb)\leq e^{-N}.
\eeq
\end{proposition}
\begin{proof}
Let $(\bm \eta(t))_{t\geq 0}$ be the strong solution of $\beta$-DBM with potential $V=0$,
\beq\label{e:DBMV0}
\rd \eta_i(t) = \sqrt{\frac{2}{\beta N}} \rd B_i(t) +\frac{1}{N}\sum_{j:j\neq i}\frac{\rd t}{\eta_i(t)-\eta_j(t)}\rd t,\quad i=1,2,\cdots, N.
\eeq
We take the initial data as $\bm\eta(0)=\bm\la(0)\in \overline{\Delta_N}$. Thanks to \cite[Lemma 4.3.17]{MR2760897}, there exists a finite constant $\fb_1=\fb_1(\fa, T)$, such that 
\beq\label{e:normbound}
\Omega\deq\{\max\{|\la_1(t)|,|\la_N(t)|\}\leq \fb_1\},\quad \bP(\Omega)\geq 1-e^{-N}.
\eeq

By taking difference of the stochastic differential equations satisfied by $(\bm \la(t))_{t\geq 0}$ and $(\bm \eta(t))_{t\geq 0}$, we get
\beq\label{e:diffDBM2}
\del_t(\la_i(t)-\eta_i(t))=\frac{1}{N}\sum_{j:j\neq i}\frac{(\la_j(t)-\eta_j(t))-(\la_i(t)-\eta_i(t))}{(\la_i(t)-\la_j(t))(\eta_i(t)-\eta_j(t))}\rd t -\frac{1}{2} V'(\la_i(t))\rd t.
\eeq
Let $i_0=\argmax_{i\in\qq{N}}\{\la_i(t)-\eta_i(t)\}$. For $i=i_0$, the first term of \eqref{e:diffDBM2} is non-positive, and thus on the event $\Omega$,
\begin{align}\begin{split}
\del_t(\la_{i_0}(t)-\eta_{i_0}(t))
\leq& -\frac{1}{2}\left(V'(\la_{i_0}(t))-V'(\eta_{i_0}(t))\right)-\frac{1}{2}V'(\eta_{i_0}(t))\\
\leq& -\frac{1}{2}\left(V'(\la_{i_0}(t))-V'(\eta_{i_0}(t))\right)+C,
\end{split}\end{align}
where $C=\max_{x\in[-\fb_1,\fb_1]}|V'(x)|/2$. Then thanks to Assumption \eqref{a:asumpV}, either $\la_{i_0}(t)-\eta_{i_0}(t)< 0$, or
$\del_t(\la_{i_0}(t)-\eta_{i_0}(t))\leq {\fK}(\la_{i_0}(t)-\eta_{i_0}(t))+C$. Therefore, it follows from Gronwall's inequality,
\beq
\max_{i\in\qq{N}}\{\la_i(t)-\eta_i(t)\}\leq \frac{C(e^{{\fK}t}-1)}{{\fK}}.
\eeq
And thus
\beq
\max_{i\in\qq{N}}\{\la_i(t)\}\leq \fb_1+\frac{C(e^{{\fK}t}-1)}{{\fK}}.
\eeq

Similarly, let $i_0=\argmin_{i\in\qq{N}}\{\la_i(t)-\eta_i(t)\}$, then either $\la_{i_0}(t)-\eta_{i_0}(t)> 0$, or
$\del_t(\la_{i_0}(t)-\eta_{i_0}(t))\geq {\fK}(\la_{i_0}(t)-\eta_{i_0}(t))-C$.  It follows from Gronwall's inequality that
$
\min_{i\in\qq{N}}\{\la_i(t)\}\geq -\fb_1-C(e^{{\fK}t}-1)/{\fK}. 
$
Proposition \ref{normbound} follows by taking $\fb=\fb_1+C(e^{{\fK}t}-1)/{\fK}$.
\end{proof}

Note that the constant $\fb$ in the previous proposition depends only on $V$ through its $C^1$ norm on the interval $[- \fb_1, \fb_1]$ and $\fK$.  Hence, if we replace $V'(x)$ by $V'(x) \chi (x)$ where $\chi$ is a smooth cut-off function on $[-2 \fb , 2\fb]$ (we assume $\fb > 1$), then by Proposition \ref{normbound} the solutions of \eqref{DBM} with the original potential $V'(x)$ and the cut-off potential $V'(x)\chi(x)$ agree with exponentially high probability.  Hence for the remainder of the paper it will suffice for our purposes to work with the cut-off potential $V'(x) \chi(x)$.

We introduce the following quasi-analytic extension of $V'$ of order three,
\beq\label{def:V'}
V'(x+\ri y)\deq\left(V'(x)\chi(x)+\ri y \del_x(V'(x)\chi(x)) -\frac{y^2}{2}\del_x^2(V'(x)\chi(x))\right)\chi(y).
\eeq
We denote,
\beq
\del_z=\frac{1}{2}(\del_x-\ri \del_y),\quad \del_{\bar z}=\frac{1}{2}(\del_x+\ri \del_y).
\eeq

We rewrite \eqref{e:limitm0} in the following
\begin{align}\label{eq:dm0}\begin{split}
\del_t m_t(z)=\del_z m_t(z)\left(m_t(z)+\frac{V'(z)}{2}\right)+\frac{m_t(z)\del_z V'(z)}{2}+\int_{\bR}g(z,x) \rd \mu_t(x),
\end{split}
\end{align}
where 
\beq \label{def:gzx}
g(z,x)\deq\frac{V'(x)-V'(z)-(x-z)\del_zV'(z)}{2(x-z)^2},\quad g(x,x)\deq\frac{V'''(x)}{4}.
\eeq

By our definition \eqref{def:V'}, $V'$ is quasi-analytic along the real axis. One can directly check the following properties of $g(z,x)$ and $V$.
\begin{proposition}\label{e:gbound}
Suppose $V$ satisfies  Assumption \ref{a:asumpV}.  Let $V'(z)$ and $g(z,x)$ be as defined in \eqref{def:V'} and \eqref{def:gzx}.  There exists a universal constant $C$ depending on $V$, such that 
\begin{enumerate}
\item $\|V'(z)\|_{C^1} \leq C$, $ | \Im [ V' (z) ] | \leq C | \Im [z]|  $ and $| \Im[\del_z V'(z)] | \leq C | \Im [ z ]|$.  
\item The following bounds hold uniformly over $z\in \bC$ and $x \in \bR$.  We have $ |g (z, x) | + | \del_x g(z, x) | \leq C$.  Furthermore, $ | \del^2_x g(z, x) | \leq C |z-x|^{-1}$ and $| \Im [ g (z, x) ] | \leq C |\Im [z]|$. 


\item If we further assume $V$ is $C^5$, then $\|V'(z)\|_{C^2} \leq C$, and uniformly over $z\in \bC$ and $x \in \bR$, $|\del_z g(z,x)|+|\del_{\bar z}g(z,x)|\leq C$.

\end{enumerate}
\end{proposition}


We define the following quasi-analytic extension of $g(z,\cdot)$ of order two,
\beq\label{def:tdg}
\td g(z,x+\ri y)\deq(g(z,x)+\ri y \del_x g(z,x))\chi(y),
\eeq 
By the Helffer-Sj{\"o}strand formula,
\beq
\int_{\bR}g(z,x) \rd \mu_t(x)=\frac{1}{\pi}\int_{\bC} \del_{\bar w} \td g(z,w) m_t(w)\rd^2 w,
\eeq
and  so we can rewrite \eqref{eq:dm0} as an autonomous differential equation of $m_t(z)$:
\beq \label{e:limitm}
\del_t m_t(z)=\del_zm_t(z)\left(m_t(z)+\frac{V'(z)}{2}\right)+\frac{m_t(z)\del_z V'(z)}{2}+\frac{1}{\pi}\int_{\bC} \del_{\bar w} \td g(z,w) m_t(w)\rd^2 w.
\eeq

\subsection{Stieltjes transform of the limit measure-valued process}
\label{sec:mztzt}
In this subsection we analyze the differential equation of the Stieltjes transform of the limiting measure-valued process
\eqref{e:limitm} with initial data $\mu_0$ which we assume to have $\supp \mu_0\in[-\fa,\fa]$. We fix a constant time $T$.  By Theorem \ref{thm:limitm} and Proposition \ref{normbound}, there exists a finite constant $\fb=\fb({\fa}, T)$ such that $\supp \mu_t\in[-\fb,\fb]$ for any $0\leq t\leq T$.

We analyze \eqref{e:limitm} by the method of characteristics. Let
\beq \label{def:zt}
\del_t z_t(u)=-m_t(z_t(u))-\frac{V'(z_t(u))}{2},\qquad z_0=u\in \bC_+,
\eeq
If the context is clear, we omit the parameter $u$, i.e., we simply write $z_t$ instead of $z_t(u)$. 

For any $\epsilon>0$, let $\bC_+^\epsilon=\{z:\Im[z]>\epsilon\}$.  Since $m_t$ is analytic, bounded and Lipschitz on the closed domain $\overline{\bC_+^\epsilon}$, we have that for any $u$ with $u\in\bC_+^\epsilon$, the solution  $z_t(u)$ exists, is unique, and is well defined before exiting the domain. Thanks to the local uniqueness of the solution curve, it follows,  by taking $\epsilon\rightarrow 0$,  that for any $u$ with $\Im[u]>0$, the solution curve $z_t(u)$ is well defined before it exits the upper half plane. For any $u\in \bC_+$, either the flow $z_t(u)$ stays in the upper half plane forever, or there exists some time $t$ such that $\lim_{s\rightarrow t}\Im[z_s(u)]=0$.

Plugging \eqref{def:zt} into \eqref{e:limitm}, and applying the chain rule we obtain 

\beq \label{e:mtzt}
\del_t m_t(z_t)=\frac{m_t(z_t)\del_z V'(z_t)}{2}+\frac{1}{\pi}\int_{\bC} \del_{\bar w} \td g(z_t,w) m_t(w)\rd^2 w.
\eeq

The behaviors of $z_s$ and $m_s(z_s)$ are governed by the system of equations \eqref{def:zt} and \eqref{e:mtzt}. The following properties of the flows $\Im[z_s]$ and $\Im[m_s(z_s)]$ will be used throughout the paper.
\begin{proposition}\label{p:propzt}
Suppose $V$ satisfies the Assumption \ref{a:asumpV}. Fix a time $T>0$. There exists a constant $C=C(V,\fb)$ such that the following holds.  For any $0\leq s\leq t\leq T$ with $\Im [ z_t ] >0$, the following estimates hold uniformly for initial $u = z_0$ in compact subsets of $ \mathbb{C}_+$.
\begin{align}
\label{dermtbound}
\Im[m_t(z)]\Im[z]\leq &1,\quad |\del_z m_t(z)|\leq \frac{\Im[m_t(z)]}{\Im[z]},\\
\label{e:msmtbound}
 e^{-C(t-s)}\Im[z_t]\leq& \Im[z_s],\\
\label{e:mtbound}
e^{-tC}\Im[m_0(z_0)]\leq &\Im[m_t(z_t)]\leq e^{tC}\Im[m_0(z_0)],\\
\label{e:ztbound}
e^{-Ct}\left(\Im[z_0]-\frac{e^{Ct}-1}{C}\Im[m_0(z_0)]\right)\leq&\Im[z_t]\leq e^{Ct}\left(\Im[z_0]-\frac{1-e^{-Ct}}{C}\Im[m_0(z_0)]\right),
\end{align}
and 
\begin{align}\label{e:propzt}
\int_{s}^t\frac{\Im[m_\tau(z_\tau)] \rd \tau}{\Im[z_\tau]}\leq C(t-s)+\log \frac{\Im[z_s]}{\Im[z_t]},\qquad \int_{s}^t\frac{\Im[m_\tau(z_\tau)]\rd \tau}{\Im[z_\tau]^p} \leq \frac{C}{\Im[z_t]^{p-1}},\quad p>1
\end{align}
\end{proposition}
\begin{proof}

The estimates \eqref{dermtbound} are general and hold for any Stieltjes transform.  First, we have
\begin{align*}
\Im[ m_t(z)]\Im[z]=\int_{\bR}\frac{\Im[z]^2\rd \mu_t(x)}{|x-z|^2}\leq \int_\bR \rd \mu_t (x)= 1,
\end{align*}
and secondly we have,
\begin{align*}
|\del_z m_t(z)|=\left|\int_{\bR}\frac{\rd \mu_t(x)}{(x-z)^2}\right|
\leq \frac{1}{\Im[z]}\int_{\bR}\frac{\Im[z]\rd \mu_t(x)}{|x-z|^2}=\frac{\Im[m_t(z)]}{\Im[z]}.
\end{align*}

Since $\Im[m_s(z_s)]\geq 0$, it follows from \eqref{def:zt} and the estimate $|\Im[V'(z_s)]|=\OO(\Im[z_s])$ of Proposition \ref{e:gbound} that there exists a constant $C$ s.t. 
\beq \label{e:dImzbound}
\del_s\Im[z_s]\leq C\Im[z_s].
\eeq
The estimate \eqref{e:msmtbound} follows.  


By Proposition \ref{e:gbound}, $\Im[V'(z)]=\OO( \Im[z])$. It follows from taking imaginary part of \eqref{def:zt} that there exists some constant $C$ depending on $V$, 
such that 
\beq
\label{e:dzt1}\left|\del_s\Im[z_s]+\Im[m_s(z_s)\right| \leq  C\Im[z_s].
\eeq
By rearranging, \eqref{e:dzt1} leads to the inequalities
\beq
\label{e:dzt2}  -e^{Cs}\Im[m_s(z_s)]\leq \del_s \left(e^{Cs}\Im[z_s]\right),
\quad \del_s \left(e^{-Cs}\Im[z_s]\right)\leq -e^{-Cs}\Im[m_s(z_s)].
\eeq

Similarly, by taking the imaginary part of \eqref{e:mtzt}, i.e.
\beq
\del_s m_s(z_s)=\frac{m_s(z_s)\del_z V'(z_s)}{2}+ \int_{\bR} g (z, x ) \d \mu_s (x),
\eeq
 and using the estimates $ | \Im[\del_z V'(z)] | + |\Im[g(z,x)]|=\OO( \Im[z])$  and $|\del_z V' (z) | \leq C$ in  Proposition \ref{e:gbound}  we obtain,
\beq
\label{e:dmt}\left|\del_s\Im[m_s(z_s)]\right|\leq \left|\Im\left[\frac{\del_zV'(z_s)m_s(z_s)}{2}\right]\right|+C'\Im[z_s]\leq C \Im[m_s(z_s)].
\eeq
In the last equality we used $\supp \mu_s\in[-\fb, \fb]$, and so
\begin{align*}
&|\Im[\del_z V'(z_s)]\Re[m_s(z_s)]|=1_{\{ |\Re[z_s]|\leq2\fb \}}\OO(\Im[z_s]|\Re[m_s(z_s)]|)\\
=&1_{ \{ |\Re[z_s]|\leq 2\fb \}}\OO\left(\int_{\bR}\frac{\Im[z_s]\Re[x-z_s]\rd \mu_s(x)}{|x-z_s|^2}\right)
=\OO\left(\int_{\bR}\frac{3\fb\Im[z_s]\rd \mu_s(x)}{|x-z_s|^2}\right)
=\OO(\Im[m_s(z_s)]).
\end{align*}
We also used $\supp \mu_s\in[-\fb, \fb]$, thus for $x\in \supp \mu_s$ and $|\Re[z_s]\leq 2\fb$, it holds $|\Re[x-z_s]|\leq 3\fb$.


The estimate \eqref{e:mtbound} then follows from \eqref{e:dmt} and Gronwall's inequality, and the estimate \eqref{e:ztbound} follows from combining \eqref{e:dzt2} and \eqref{e:mtbound}. For \eqref{e:propzt} we have by \eqref{e:dzt2},
\beq
\int_{s}^t\frac{\Im[m_\tau(z_\tau)] \rd \tau}{\Im[z_\tau]^p}\leq 
\int_{s}^t\frac{-\del_\tau \left(e^{-C\tau}\Im[z_\tau]\right) \rd \tau}{\left(e^{-C\tau}\Im[z_\tau]\right)^p}.
\eeq
The case $p=1$ follows. For $p > 1$, we have
\beq
\int_{s}^t\frac{\Im[m_\tau(z_\tau)] \rd \tau}{\Im[z_\tau]^p}\leq \frac{1}{p-1}\left(\frac{1}{\left(e^{-Ct}\Im[z_t]\right)^{p-1}}-\frac{1}{\left(e^{-Cs}\Im[z_s]\right)^{p-1}}\right)\leq \frac{C'}{\Im[z_t]^{p-1}}.
\eeq

\end{proof}

We have the following result for the flow map $u \to z_t (u)$.
\begin{proposition}\label{prop:invflow}
Suppose that $V$ satisfies Assumption \ref{a:asumpV}. Fix a time $T$. For any $0\leq t\leq T$, there exists an open domain $\Omega_t\subset\bC_+$, such that the vector flow map $u\mapsto z_t(u)$ is a $C^1$ homeomorphism from $\Omega_t$ to $\bC_+$. 
\end{proposition}
\begin{proof}
We define 
\begin{align}
\Omega_t\deq  \{u\in \bC_+: z_s(u)\in \bC_+, 0\leq s\leq t\}.
\end{align}
By Assumption \ref{a:asumpV}, $V$ is a $C^4$ function. From our construction of $V'(z)$ as in \eqref{def:V'}, $V'(z)$ is a $C^1$ function. Thus the vector flow map $u\mapsto z_t(u)$ is a $C^1$ map from $\Omega_t$ to $\bC_+$. We need to show that it is invertible. Define the following flow map by
\beq\label{e:invflow}
\del_sy_s(v)=m_{t-s}(y_{s}(v))+\frac{V'(y_{s}(v))}{2}, \quad y_0=v\in \bC_+.
\eeq
for $0 \leq s \leq t$.
Since $\Im[m_{t-s}(y_s(v))]\geq 0$, there exists some constant $C$ depending on $V$ 
\beq
\del_s \left( e^{Cs}\Im[y_s]\right)\geq 0.
\eeq
Therefore $y_s(v)$ is well defined for $0\leq s\leq t$, and it will stay in $\bC_+$.  Furthermore, $v\mapsto y_t(v)$ is a $C^1$ map, and is  the inverse of $u\mapsto z_t(u)$. 
\end{proof}

\section{Rigidity of $\beta$-DBM}\label{s:rigidity}

In this section we prove the local law and optimal rigidity for $\beta$-DBM with general initial data.  Let $\mu_t$ be the unique solution of \eqref{e:MVeq} with initial data
\beq
\mu_0 (x) := \frac{1}{N} \sum_{i=1}^N \delta_{\lambda_i (0) } (x).
\eeq
Denote by $m_t$ its Stieltjes transform.
  We introduce some notation used in the statement and proof of the local law.  We fix a small parameter $\delta$, a large constant $\fc\geq1$, a large constant $K$, and control parameter $M=(\log N)^{2+2\delta}$. For any time $s\ll1 $, we define the spectral domain, 
\beq\label{def:dom}
\dom_s=\left\{w\in \bC_+: \Im[w]\geq \frac{e^{Ks}M\log N}{N\Im [m_s(w)]}\vee \frac{e^{Ks}}{N^{\fc}},\quad \Im[w]\leq 3{\fb}-s,\quad  |\Re[w]|\leq 3\fb-s\right\}.
\eeq

The following is the local law for $\beta$-DBM.
\begin{theorem}\label{t:rigidity}
Suppose $V$ satisfies the Assumption \ref{a:asumpV}. Fix $T=(\log N)^{-2}$. Let $\beta\geq 1$ and assume that the initial data satisfies $-\fa\leq \la_1(0)\leq \la_2(0)\cdots\leq \la_N(0)\leq \fa$ for a fixed $\fa >0$.  Uniformly for any $0\leq t\ll T$, and $w\in \dom_t$ 
the following estimate holds with overwhelming probability,
\beq \label{e:diffmm}
|\td m_t(w)- m_t(w)|\leq \frac{M}{N\Im[w]}.
\eeq
\end{theorem}

The following rigidity estimates are a consequence of the local law.
\begin{corollary}\label{c:eigloc}
Under the assumptions of Theorem \ref{t:rigidity}. Fix time $T=(\log N)^{-2}$. With overwhelming probability, uniformly for any $0\leq t\leq T$ and $i\in \qq{1, N}$, we have
\beq\label{e:eigloc}
\gamma_{i-CM\log N}(t)-N^{-\fc+1}\leq \lambda_i(t) \leq \gamma_{i+CM\log N}(t)+N^{-\fc+1}
\eeq
where $\fc$ is any large constant, $\gamma_i(t)$ is the classical particle location at time $t$,
\beq\label{e:classlocation}
\gamma_i(t)=\sup_{x}\left\{\int_{-\infty}^x \rd \mu_t(x)\geq \frac{i}{N}\right\},\quad i\in\qq{1,N}. 
\eeq
We make the convention that $\gamma_i(t)=-\infty$ if $i<0$, and $\gamma_i(t)=+\infty$ if $i>N$.
\end{corollary}

We will prove Theorem \ref{t:rigidity} at the end of Section \ref{sec:locallawproof}.  The proof of Corollary \ref{c:eigloc} is standard and is given in Section \ref{sec:corproof}.

\begin{remark}\label{r:imagpart}
Notice that 
\begin{align*}
\eta\mapsto\eta\Im[m_t(E+\ri\eta)]=\int_{\bR}\frac{\eta^2\rd\mu_t(x)}{(E-x)^2+\eta^2},
\end{align*}
 is a monotonically increasing function. Similarly, $\eta\mapsto\eta\Im[\tilde m_t(E+\ri\eta)]$ is monotonic.
 
 We now prove the following deterministic fact.  Suppose that the estimate \eqref{e:diffmm} holds on $\dom_t$.  We claim that under this assumption the estimate
 \beq\label{e:imagbound}
|\Im[\td m_t(w)]- \Im[m_t(w)]|\leq \frac{3e^{Kt}M\log N}{N\Im[w]}.
\eeq
holds on the larger domain $ w = E + \i \eta$ with $|E| \leq 3 \fb - t $ and $e^{ K t } N^{- \fc} \leq \eta \leq 3 \fb - t$. 
 
Let
\beq\label{defetax}
\eta(E)=\inf_{\eta\geq 0}\{\eta\Im[m_t(E+\ri\eta)]\geq e^{Kt}M\log N/N\}.
\eeq
By the assumption \eqref{e:diffmm} and the definition of $\dom_t$ we only need to check the case that $ \eta (E) > e^{ Kt } N^{ - \fc}$ and $\eta < \eta (E)$.
In this case we have  $\eta(E)\Im[m_t(E+\ri\eta(E))]=e^{Kt}M\log N/N$,
and $\Im[m_t(w)]\leq e^{Kt}M\log N/N\eta$, and so
\begin{align*}
|\Im[\td m_t(w)]- \Im[m_t(w)]|
\leq& \Im[\td m_t(w)]+\frac{e^{Kt}M\log N}{N\eta}\\
\leq& \frac{\eta(E)}{\eta}\Im[\td m_t(E+\ri\eta(E))]+\frac{e^{Kt}M\log N}{N\eta}\\
\leq& \frac{\eta(E)}{\eta}\left|\Im[\td m_t(E+\ri\eta(E))]-\Im[m_t(E+\ri \eta(E)]\right|+\frac{2e^{Kt}M\log N}{N\eta}\\
\leq& \frac{3e^{Kt}M\log N}{N\eta}.
\end{align*}
In the second inequality we used monotonicity of $\eta\Im[\tilde m_t(E+\ri\eta)]$.
\end{remark}

\subsection{Properties of the spectral domain}
In this section, we prove some properties of the spectral domain $\cal D_s$ as defined in \eqref{def:dom}, which will be used throughout the proof of Theorem \ref{t:rigidity}.

\begin{lemma}
Suppose $V$ satisfies Assumption \ref{a:asumpV}. Fix a time $T$. For any $0\leq t\leq T$ such that $z_t\in \dom_t$, we have
\beq\label{e:z0ztbound}
\Im[u] \leq C N \Im[z_t(u)].
\eeq
\end{lemma}
\begin{proof}
Notice that from \eqref{dermtbound}, $\Im[u]\leq \Im[m_0(u)]^{-1}$, and by our assumption $z_t\in \dom_t$ and \eqref{e:mtbound},
\beq
\Im[z_t(u)]\geq (N\Im[m_t(z_t)])^{-1}\geq (CN\Im[m_0(z_0)])^{-1}\geq (CN)^{-1}\Im[u],
\eeq
which yields \eqref{e:z0ztbound}.
\end{proof}


\begin{proposition}\label{p:dom}
Suppose $V$ satisfies the Assumption \ref{a:asumpV}. Fix  time $T$. If for some $t\in [0,T]$, $z_t\in \dom_t$, then for any $s\in [0, t]$, $z_s\in \dom_s$.
\end{proposition}
\begin{proof}
By \eqref{e:msmtbound}, we have $\Im[z_s]\geq e^{-C(t-s)}\Im[z_t]$.  Therefore if we have $\Im[z_t]\geq e^{Kt}N^{-\fc}$ then $\Im[z_s]\geq e^{Ks}N^{-\fc}$ as long as we take $K\geq C$.

Combining $\del_s\Im[z_s]\leq C\Im[z_s]$ from \eqref{e:dImzbound}, with \eqref{e:dmt}, 
yields that there is a  constant $C'$ so that 
\beq
\del_s\left(\Im[z_s]\Im[m_s(z_s)]\right)\leq C'\Im[z_s]\Im[m_s(z_s)].
\eeq
Therefore,
$\Im[z_s]\Im[m_s(z_s)]\geq e^{-C'(t-s)}\Im[z_t]\Im[m_t(z_t)]$ and so if
$\Im[z_t]\Im [m_t(z_t)]\geq e^{Kt}M\log N/N$, then $\Im[z_s]\Im [m_s(z_s)]\geq e^{Ks}M\log N/N$, provided $K\geq C'$.

Finally, we must prove that if $\Im[z_t]\leq 3{\fb}-t$ and  $|\Re[z_t]|\leq 3\fb-t$, then for any $s\in [0,t]$, we have $\Im[z_s]\leq 3{\fb}-s$ and  $|\Re[z_s]|\leq 3\fb-s$. 
First, suppose for a contradiction that there exists $s$ such that, say, $|\Re[z_s]|> 3{\fb}-s$. By symmetry, say $\Re[z_s]> 3{\fb}-s$. Let $\tau=\inf_{\sigma\geq s}\{\Re[z_\sigma]\leq 3{\fb}-\sigma\}$; then $\tau\leq t$. For any $\sigma\in [s,\tau]$, $\Re[z_\sigma]\geq 3\fb -T\geq 2\fb$, we have $V'(z_\sigma)=0$, and therefore $|\del_\sigma z_{\sigma}|\leq |m_\sigma(z_{\sigma})|\leq \dist(z_\sigma, \supp \mu_\sigma)^{-1}$. Recall that we have chosen $\fb$ large so that, $\supp \mu_t\in [-\fb, \fb]$. Therefore $|\del_\sigma z_{\sigma}|\leq \fb^{-1}$ and $\Re[z_\tau]\geq \Re[z_s]-(\tau-s)/\fb> 3\fb-\tau$, as long as we take $\fb > 1$.  Therefore we derive a contradiction.  A similar argument applies to the case that 
 $\Im[z_s]> 3{\fb}-s$. This finishes the proof of Proposition \ref{p:dom}.
\end{proof}

We have the following weak control on the $C^1$ norm of the flow map $u\mapsto z_t(u)$. A much stronger version will be proved in Proposition \ref{prop:dzms}.
\begin{proposition}\label{prop:dztbound0}
Suppose $V$ satisfies the Assumption \ref{a:asumpV}. Fix time $T$. For any $0\leq t\leq T$ with $z_t\in \dom_t$, we have with $u = x + \i y$,
\beq\label{e:dztbound0}
|\del_x z_t(u)| + |\del_y z_t(u)|=\OO(N),
\eeq
where the implicit constant depends on $T$ and $V$.
\end{proposition}
\begin{proof}
From Proposition \ref{prop:invflow}, we know that $u\mapsto z_t(u)$ is a $C^1$ map. By differentiating both sides of \eqref{def:zt}, we get
\beq
\del_s\del_xz_s(u)=-\del_z m_s(z_s(u))\del_xz_s(u) -\frac{\del_z V'(z_s(u))\del_x z_s(u)+\del_{\bar z}V'(z_s(u))\del_x\bar z_s(u)}{2}.
\eeq
It follows that 
\begin{align}\begin{split}
&\del_s|\del_xz_s(u)|^2=2\Re[\del_s \del_x z_s(u) \del_x\bar z_s(u)]\\
=&-2\Re[\del_z m_s(z_s(u))]|\del_xz_s(u)|^2 -\Re[\del_z V'(z_s(u))|\del_x z_s(u)|^2+\del_{\bar z}V'(z_s(u))(\del_x\bar z_s(u))^2].
\end{split}\end{align}

By Proposition \ref{p:propzt} we have $|\del_z m_s(z_s(u))|\leq \Im[m_{s}(z_s(u))]/\Im[z_s(u)]$, and by Proposition \ref{e:gbound} we have  $|\del_{z}V'(z_s(u))|, |\del_{\bar z}V'(z_s(u))|=\OO(1)$.  Therefore,
\beq
\del_s|\del_xz_s(u)|^2\leq 2\left(\frac{\Im[m_s(z_s(u))]}{\Im[z_s(u)]}+C\right)|\del_xz_s(u)|^2. 
\eeq
Since $z_0(u)=u$, by Gronwall's inequality and \eqref{e:propzt} of Proposition \ref{p:propzt}, we have
\beq
|\del_xz_t(u)|^2\leq \exp\left( 2 \int_0^t \left(\frac{\Im[m_s(z_s(u))]}{\Im[z_s(u)]}+C\right)\rd s\right)\leq  \left( \frac{e^{Ct}\Im[u]}{\Im[z_t(u)]} \right)^2 \leq CN^2,
\eeq
where we used \eqref{e:z0ztbound}.
It follows that $|\del_x z_t(u)|=\OO(N)$. The estimate for $|\del_y z_t(u)|$ follows from the same argument. 
\end{proof}

We define the following lattice on the upper half plane $\bC_+$, 
\beq\label{def:L}
\cal L=\left\{E+\ri \eta\in \dom_0: E\in \bZ/ N^{(3\fc+1)}, \eta\in \bZ/N^{(3\fc+1)}\right\}.
\eeq

Thanks to Propositions \ref{prop:invflow} and \ref{prop:dztbound0}, we have the following.
\begin{proposition}\label{prop:continuity}
Suppose $V$ satisfies Assumption \ref{a:asumpV}. Fix a time $T$. For any $0\leq t\leq T$ and $w\in \dom_t$, there exists some lattice point $u\in \cal L\cap z_t^{-1}(\dom_t)$, such that
\beq
|z_t(u)-w|=\OO(N^{-3\fc}),
\eeq
where the implicit constant depends on $T$ and $V$.
\end{proposition}

\subsection{Proof of Theorem \ref{t:rigidity}} \label{sec:locallawproof}

In this section we prove \eqref{e:diffmm}. By Proposition \ref{prop:invflow}, the flow map $u\mapsto z_t(u)$ is a surjection from $\Omega_t$ (as defined in Proposition \ref{prop:invflow}) to the upper half plane $\bC_+$.   We want to prove the estimate
\begin{align}\label{e:conti}
\left|\td m_t(z_t)-m_t(z_t)\right|\leq \frac{M}{N\Im[z_t]},
\end{align}
for $z_t\in \dom_t$.

By Ito's formula, $\td m_s(z)$ satisfies the stochastic differential equation
\begin{align}\begin{split}\label{eq:dm}
\rd \td m_s(z)= -&\sqrt{\frac{2}{\beta N^3}}\sum_{i=1}^N \frac{{\rm d} B_i(s)}{(\la_i(s)-z)^2}+\td m_s(z)\del_z \td m_s(z)\rd s \\
+&\frac{1}{2N}\sum_{i=1}^{N}\frac{V'(\la_i(s))\rd t}{(\la_i(s)-z)^2}+\frac{2-\beta}{\beta N^2}\sum_{i=1}^{N}\frac{\rd s}{(\lambda_i(s)-z)^3}.
\end{split}\end{align}
%
We can rewrite \eqref{eq:dm} as 
\begin{align}\label{eq:dm3}\begin{split}
\rd \td m_s(z)=-&\sqrt{\frac{2}{\beta N^3}}\sum_{i=1}^N \frac{{\rm d} B_i(s)}{(\la_i(s)-z)^2}+\del_z \td m_s(z)\left(\td m_s(z)+\frac{V'(z)}{2}\right)+\frac{\td m_s(z)\del_z V'(z)}{2}\\+& \frac{1}{\pi}\int_{\bC} \del_{\bar w} \td g(z,w) \td m_s(w)\rd^2 w
+\frac{2-\beta}{\beta N^2}\sum_{i=1}^{N}\frac{\rd s}{(\la_i(s)-z)^3},
\end{split}
\end{align}
where $V'(z)$ and $\td g(z,w)$ are defined in \eqref{def:V'} and \eqref{def:tdg} respectively.
Plugging \eqref{def:zt} into \eqref{eq:dm3}, and by the chain rule,  we have 
\begin{align}\label{e:tdmzt}\begin{split}
\rd \td m_s(z_s)=-&\sqrt{\frac{2}{\beta N^3}}\sum_{i=1}^N \frac{{\rm d} B_i(s)}{(\la_i(s)-z_s)^2}+\del_z \td m_s(z_s)\left(\td m_s(z_s)-m_s(z_s)\right)+\frac{\td m_s(z_s)\del_zV'(z_s)}{2}\\
+ &\frac{1}{\pi}\int_{\bC} \del_{\bar w} \td g(z_s,w) \td m_s(w)\rd^2 w \rd t
+\frac{2-\beta}{\beta N^2}\sum_{i=1}^{N}\frac{\rd s}{(\la_i(s)-z_s)^3}.
\end{split}
\end{align}
It follows by taking the difference of \eqref{e:mtzt} and \eqref{e:tdmzt} that, 
\begin{align}\label{e:diffm}\begin{split}
\rd (\td m_s(z_s)-m_s(z_s))=-&\sqrt{\frac{2}{\beta N^3}}\sum_{i=1}^N \frac{{\rm d} B_i(s)}{(\la_i(s)-z_s)^2}+\left(\td m_s(z_s)-m_s(z_s)\right)\del_z \left(\td m_s(z_s)+\frac{V'(z_s)}{2}\right)\rd s\\
+& \frac{1}{\pi}\int_{\bC} \del_{\bar w} \td g(z_s,w) (\td m_s(w)-m_s(w))\rd^2 w\rd s
+\frac{2-\beta}{\beta N^2}\sum_{i=1}^{N}\frac{\rd s}{(\la_i(s)-z_s)^3}.
\end{split}
\end{align}
Using the fact that $m_0 (z) = \td m_0 (z)$, we can integrate both sides of \eqref{e:diffm} from $0$ to $t$ and obtain
\beq \label{eq:mzt}
\td m_t(z_t)-m_t(z_t)=\int_0^t \cE_1(s)\rd s+\rd \cE_2(s),
\eeq
where the error terms are
\begin{align}
\label{defcE1}\cE_1(s)=&\left(\td m_s(z_s)-m_s(z_s)\right)\del_z \left(\td m_s(z_s)+\frac{V'(z_s)}{2}\right)+
\frac{1}{\pi}\int_{\bC} \del_{\bar w} \td g(z_s,w) (\td m_s(w)-m_s(w))\rd^2 w , \\
\label{defcE2}\rd \cE_2(t)=&\frac{2-\beta}{\beta N^2}\frac{\rd s}{(\lambda_i(s)-z_s)^3}-\sqrt{\frac{2}{\beta N^3}}\sum_{i=1}^N \frac{{\rd} B_i(s)}{(\la_i(s)-z_s)^2}.
\end{align}
We remark that $\cal E_1$ and $\cal E_2$ implicitly depend on $u$, the initial value of the flow $z_s(u)$.  The local law will eventually follow from an application of Gronwall's inequality to \eqref{eq:mzt}.

We define the stopping time 
\beq\label{stoptime}
\sigma\deq \inf_{s\geq0}\left\{\exists w\in \dom_s: \left|\td m_s(w)-m_s(w)\right|\geq \frac{M}{N\Im[w]}\right\}\wedge t.
\eeq
In the rest of this section we prove that with overwhelming probability we have $\sigma=t$. Theorem \ref{t:rigidity} follows.

For any lattice point $u\in \cal L$ as in \eqref{def:L}, we denote 
\beq\label{def:tu}
t(u)= \sup_{s\geq0}\{ z_s(u)\in \dom_s\}\wedge t.
\eeq
By Proposition \ref{p:dom} we have that  $z_s(u)\in\dom_s$ for any $0\leq s\leq t(u)$. We decompose the time interval $[0,t(u)]$ in the following way.  First set $t_0=0$, and define
\beq
t_{i+1}(u):=\sup_{s\geq t_i(u)}\left\{ \Im[z_s(u)] \geq \frac{\Im[z_{t_i}(u)]}{2}\right\}\wedge t(u),\quad i=0,1,2,\cdots.
\eeq
By \eqref{e:z0ztbound}, there exists some constant $C$ depending on $V$, 
 such that $\Im[z_0(u)]\leq CN\Im[z_{t(u)}(u)]$, and thus the above sequence will terminate at some $t_k(u)=t(u)$ for $k=\OO(\log N)$ depending on $u$, the initial value of $z_s$. 
Moreover, by \eqref{e:msmtbound}, 
for any $t_i(u)\leq s_1\leq s_2\leq t_{i+1}(u)$,
\beq\label{e:Imzsztbound}
e^{-CT}\leq e^{-C(s_2-s_1)}\leq \frac{\Im[z_{s_1}(u)]}{\Im[z_{s_2}(u)]}\leq \frac{e^{C(s_1-t_i)}\Im[z_{t_{i}}(u)]}{e^{C(s_2-t_{i+1})}\Im[z_{t_{i+1}}(u)]}\leq 2e^{CT}.
\eeq

%
%
%
%
%
%

We first derive an estimate of $\int \rd \cE_2(s)$ in terms of $\{m_s(z_s(u)), 0\leq s\leq t(u)\}$. 

\begin{proposition}\label{p:error}
Under the assumptions of Theorem \ref{t:rigidity}. There exists an event $\Omega$ that holds with overwhelming probability on which we have for every $0 \leq \tau \leq t (u)$ and $u \in \cal L$,
\beq\label{eq:estet}
 \left|\int_0^{\tau\wedge \sigma} \rd \cE_2(s)\right|\leq \frac{C(\log N)^{1+\delta}}{N\Im[z_{\tau\wedge\sigma}(u)]}.
\eeq
\end{proposition}

\begin{proof}

For simplicity of notation, we write $t_i = t_i (u)$ and $z_s = z_s (u)$. For any $s\leq t_i$, by our choice of the stopping time $\sigma$ (as in \eqref{stoptime}), and the definition of domain $\dom_s$ (as in 
\eqref{def:dom}), we have
\beq\label{e:replacebound}
\Im[\td m_{s\wedge \sigma}(z_{s\wedge \sigma})]\leq 2\Im[m_{s\wedge \sigma}(z_{s\wedge \sigma})].
\eeq

 For the first term in \eqref{defcE2} we have
\begin{align}\begin{split}\label{eq:1term}
\sup_{0\leq \tau\leq t_i}\left|\frac{2-\beta}{\beta N^2}\int_0^{\tau\wedge\sigma}\sum_{i=1}^{N}\frac{\rd s}{(\lambda_i(s)-z_s)^3}\right|
\leq& \frac{2-\beta}{\beta N^2}\int_0^{t_i\wedge \sigma}\sum_{i=1}^{N}\frac{\rd s}{|\lambda_i(s)-z_s|^3}\\
\leq& \frac{C}{ N^2}\int_0^{t_i\wedge \sigma}\sum_{i=1}^{N}\frac{\rd s}{\Im[z_s]|\lambda_i(s)-z_s|^2}
= C \int_0^{t_i\wedge\sigma}\frac{\Im[\td m_s(z_s)]\rd s}{N\Im[z_s]^2}\\
\leq&  C \int_0^{t_i\wedge\sigma}\frac{2\Im[m_s(z_s)]\rd s}{N\Im[z_s]^2}
\leq \frac{C'}{\Im[z_{t_i\wedge \sigma}]},
\end{split}
\end{align}
 where we used \eqref{e:replacebound} and \eqref{e:propzt}.
 
For the second term in \eqref{defcE2} we have
\begin{align}\begin{split}
&\left\langle\sqrt{\frac{2}{\beta N^3}}\int_0^{\cdot\wedge\sigma}\sum_{i=1}^N \frac{{\rm d} B_i(s)}{(\la_i(s)-z_s)^2} \right\rangle_{t_i}
= \frac{2}{\beta N^{3/2}}\int_0^{t_i\wedge\sigma}\sum_{i=1}^N \frac{ \rd s}{|\la_i(s)-z_s|^4}\\
\leq& \frac{2}{\beta N^{3/2}}\int_0^{t_i\wedge\sigma}\sum_{i=1}^N \frac{ \rd s}{\Im[z_s]^2|\la_i(s)-z_s|^2}
=\frac{2}{\beta}\int_0^{t_i\wedge\sigma}\frac{\Im[\td m_s(z_s)]\rd s}{N^2\Im[z_s]^3}\\
\leq&  
\frac{2}{\beta}\int_0^{t_i\wedge\sigma}\frac{2\Im[m_s(z_s)]\rd s}{N^2\Im[z_s]^3}
\leq \frac{C}{N^2\Im[z_{t_i\wedge \sigma}]^2},
\end{split}
\end{align} 
again we used \eqref{e:replacebound} and \eqref{e:propzt}.
Therefore, by Burkholder-Davis-Gundy inequality, for any $u\in \cal L$ and $t_i$, the following holds with overwhelming probability, i.e., $1-C\exp\{ -c(\log N)^{1+\delta}\}$,
\begin{align}\begin{split}\label{eq:2term}
&\sup_{0\leq \tau\leq t_i}\left|\sqrt{\frac{2}{\beta N^3}}\int_0^{\tau\wedge\sigma}\sum_{i=1}^N \frac{{\rm d} B_i(s)}{(\la_i(s)-z_s)^2} \right|
\leq \frac{C(\log N)^{1+\delta}}{N\Im[z_{t_i\wedge \sigma}]}.
\end{split}
\end{align}
We define $\Omega$ to be the set of Brownian paths $\{B_1(s), \cdots, B_N(s)\}_{0\leq s\leq t}$ on which the following two estimates hold.  
\begin{enumerate}
\item First we have, $-\fb\leq \lambda_1(s)\leq \lambda_2(s)\leq \cdots \leq \la_N(s)\leq \fb$ uniformly for all $s\in [0,T]$.
\item Second for any $u\in \cal L$ and $i=1,2,\cdots, k$,  \eqref{eq:2term} holds,
\end{enumerate}

It follows from Proposition \ref{normbound} and the discussion above,  $\Omega$ holds with overwhelming probability, i.e., $\bP(\Omega)\geq 1-C|\cal L|(\log N)\exp\{ -c(\log N)^{1+\delta}\}$.

Therefore, for any $\tau\in[t_{i-1},t_i]$, the bounds \eqref{eq:1term} and \eqref{eq:2term} yield
\begin{align}\begin{split}\label{e:continuityarg}
\left|\int_0^{\tau\wedge \sigma} \rd \cE_2(s)\right|
\leq& \frac{C'(\log N)^{1+\delta}}{N\Im[z_{t_i\wedge\sigma}]}\leq  \frac{C(\log N)^{1+\delta}}{N\Im[z_{\tau\wedge\sigma}]},
\end{split}
\end{align}
 where we used our choice of $t_i$'s, i.e. \eqref{e:Imzsztbound}.
\end{proof}

We now bound the second term of \eqref{defcE1}. 
\begin{proposition}\label{prop:HFbound}
Under the assumptions of Theorem \eqref{t:rigidity}, for any $u\in \cal L$ and $s\in[0,t(u)]$ (as in \eqref{def:tu}), we have
\beq\label{e:HFbound}
\left|\frac{1}{\pi}\int_{\bC} \del_{\bar w} \td g(z_{s\wedge\sigma}(u),w) (\td m_{s\wedge \sigma}(w)-m_{s\wedge \sigma}(w))\rd^2 w\right|\leq \frac{CM(\log N)^{2}}{N},
\eeq
where the constant $C$ depends on $V$.
\end{proposition}
\begin{proof} 
First, we note that by Proposition \ref{normbound} we may replace $g$ by $g_1 (z, x) := g (z, x) \chi (x)$ and the quasi-analytic extension by $\td g_1 (z, x + \i y ) := ( g_1 (z, x) + \i y \del_x g_1 (z, x) ) \chi (y)$.

The proof follows the same argument as \cite[Lemma B.1]{MR2639734}.
Let $S(x+\ri y)=\td m_{s\wedge \sigma}(x+\ri y)-m_{s\wedge \sigma}(x+\ri y)$.
we have
\begin{align}\begin{split}\label{e:HF}
&\left|\frac{1}{\pi}\int_{\bC} \del_{\bar w} \td g_1(z_{s\wedge\sigma}(u),w) (\td m_{s\wedge \sigma}(w)-m_{s\wedge \sigma}(w))\rd^2 w\right|
=\left|\int_{\bR}  g_1(z_{s\wedge\sigma}(u),x) (\rd \tilde \mu_{s\wedge \sigma}(x)-\rd  \mu_{s\wedge \sigma}(x))\right|\\
\leq&
C\int_{\bC_+} (|g_1(z_{s\wedge\sigma},x)|+y|\del_xg_1(z_{s\wedge\sigma}(u),x)|)|\chi'(y)||S(x+\ri y)|\rd x\rd y\\
+&C\int_{\bC_+} y\chi(y)|\del_{x}^2g_1(z_{s\wedge\sigma}(u),x)||\Im[S(x+\ri y)]|\rd x\rd y.
\end{split}\end{align} 
We start by handling the first term on the RHS of \eqref{e:HF}.  The integrand is supported in $\{x + \i y : |x| \leq 2 \fb , \fb \leq |y| \leq 2 \fb \} \subseteq \dom_t$ for every $t$. In this region we have from Proposition \ref{e:gbound} that $g_1$ and $\del_x g_1$ are bounded, and by the definition of $\sigma$ we have  that $|S(x+\ri y)|\leq M/N$  in this region and so
\beq
\int_{\bC_+} (|g_1(z_{s\wedge\sigma},x)|+y|\del_xg_1(z_{s\wedge\sigma}(u),x)|)|\chi'(y)|S(x+\ri y)|\rd x\rd y \leq \frac{CM}{N}.
\eeq
We now handle the second term on the RHS of \eqref{e:HF}.  
By the definition of $t(u)$, $z_{s\wedge\sigma}(u)\in \dom_0$, and so $\Im[z_{s\wedge\sigma}(u)]\geq N^{-\fc}$. From Proposition \ref{e:gbound} 
we have $ | \del^2_x g(z,x) | \leq C |z-x|^{-1}$, and $|\del_x^2 g_1(z,x)|\leq C\chi(x)|z-x|^{-1}$. 


We split the second integral on the righthand side of \eqref{e:HF} into the two regions $\Lambda\deq\{E+\ri \eta: 0\leq \eta\leq e^{K(s\wedge\sigma)}N^{-\fc}\}$ and $\bC_+\setminus \Lambda$.
%
%
On $\Lambda$, we use the trivial bound $\Im[S(x+\ri y)]\leq |\td m_{s\wedge \sigma}(x+\ri y)|+|m_{s\wedge \sigma}(x+\ri y)|\leq 2/y$, and obtain
\begin{align}\begin{split}
&\int_{\Lambda} y\chi(y)|\del_{x}^2g_1(z_{s\wedge\sigma}(u),x)||\Im[S(x+\ri y)]|\rd x\rd y\\
\leq& C\int_{0\leq y\leq  e^{K(s\wedge\sigma)}N^{-\fc}} \frac{\chi(x)}{|z_{s\wedge\sigma}(u)-x|}\rd x\rd y
\leq \frac{C\log N}{N^{\fc}}.
\end{split}\end{align}
On $\bC_+\setminus\Lambda$, by Remark \ref{r:imagpart}, $|\Im[S(x+\ri y)]|\leq3e^{Kt}M\log N/(Ny)$, and therefore,
\begin{align}\begin{split}
&\int_{\bC_+\setminus\Lambda} y\chi(y)|\del_{x}^2g_1(z_{s\wedge\sigma}(u),x)||\Im[S(x+\ri y)]|\rd x\rd y\\
\leq& \frac{CM\log N}{N}\int_{\bC_+} \frac{\chi(x)\chi(y)}{|z_{s\wedge\sigma}(u)-x|}\rd x\rd y
\leq \frac{CM(\log N)^{2}}{N}.
\end{split}\end{align}
This completes the proof of \eqref{e:HFbound}.
\end{proof}

%
%
%
%

\begin{proof}[Proof of Theorem \ref{t:rigidity}]
We can now start analyzing \eqref{eq:mzt}. For any lattice point $u\in \cal L$ and $\tau\in[0,t(u)]$ (as in \eqref{def:tu}), by Proposition \ref{p:error} and \ref{prop:HFbound}, we have 
\begin{align*}
\left|\td m_{\tau\wedge\sigma}(z_{\tau\wedge\sigma}(u))-m_{\tau\wedge\sigma}(z_{\tau\wedge\sigma}(u))\right|
\leq& \int_0^{\tau\wedge\sigma}\left|\td m_s(z_s(u))-m_s(z_s(u))\right|\left|\del_z \left(\td m_s(z_s(u))+\frac{V'(z_s(u))}{2}\right)\right|\rd s\\+&
\frac{C({\tau\wedge\sigma})M(\log N)^{2}}{N}
+\frac{C(\log N)^{1+\delta}}{N\Im[z_{\tau\wedge\sigma}(u)]}.
\end{align*}
Notice that for $s\leq \tau\wedge \sigma$,
\begin{align}\label{e:aterm}
\left|\del_z \left(\td m_s(z_s(u))+\frac{V'(z_s(u))}{2}\right)\right|
\leq \frac{\Im[\tilde m_s(z_s(u))]}{\Im[z_s(u)]}+C,
\end{align}
where we used \eqref{dermtbound} and $\del_z V'(z_s(u))=\OO(1)$ from Proposition \ref{e:gbound}. Since $z_s(u)\in \cal D_s$, by the definition of $\cal D_s$, we have $\Im[m_s(z_s(u))]\geq M\log N/ (N\Im[z_s(u)])$. Moreover, since $s\leq \sigma$, we have $|\tilde m_s(z_s(u))-m_s(z_s(u))|\leq M/(N\Im[z_s(u)])\leq \Im[m_s(z_s(u))]/\log N$.  Therefore,
\begin{align}\label{e:aterm}
\left|\del_z \left(\td m_s(z_s(u))+\frac{V'(z_s(u))}{2}\right)\right|
\leq \frac{\Im[\tilde m_s(z_s(u))]}{\Im[z_s(u)]}+C\leq \left(1+\frac{1}{\log N}\right)\frac{\Im[ m_s(z_s(u))]}{\Im[z_s(u)]}+C.
\end{align}
We denote
\begin{align} \label{eqn:betabd}
\beta_s(u)\deq \left(1+\frac{1}{\log N}\right)\frac{\Im[ m_s(z_s(u))]}{\Im[z_s(u)]}+C= \OO\left(\frac{\Im[ m_s(z_s(u))]}{\Im[z_s(u)]}\right).
\end{align}
We have derived the inequality,
\begin{align}\begin{split}
\left|\td m_{\tau\wedge\sigma}(z_{\tau\wedge\sigma}(u))-m_{\tau\wedge\sigma}(z_{\tau\wedge\sigma}(u))\right|
\leq& \int_0^{\tau\wedge\sigma}\beta_s(u)\left|\td m_s(z_s(u))-m_s(z_s(u))\right|\rd s\\+&
\frac{C({\tau\wedge\sigma})M(\log N)^{2}}{N}
+\frac{C(\log N)^{1+\delta}}{N\Im[z_{\tau\wedge\sigma}(u)]}.
\end{split}
\end{align}
By Gronwall's inequality, this implies the estimate
\begin{align}\begin{split}\label{e:midgronwall}
&\left|\td m_{t\wedge\sigma}(z_{t\wedge\sigma}(u))-m_{t\wedge\sigma}(z_{t\wedge\sigma}(u))\right|
\leq \frac{C({t\wedge\sigma})M(\log N)^{2}}{N}
+\frac{C(\log N)^{1+\delta}}{N\Im[z_{t\wedge\sigma}(u)]}\\
+&\int_0^{t\wedge\sigma}\beta_s(u)\left(\frac{sM(\log N)^{2}}{N}
+\frac{C(\log N)^{1+\delta}}{N\Im[z_s(u)]}\right)e^{\int_s^{t\wedge\sigma} \beta_\tau(u)\rd \tau}\rd s.
\end{split}
\end{align}

 By \eqref{e:propzt} of Proposition \ref{p:propzt}, and \eqref{eqn:betabd}, we have
\begin{align*}
e^{\int_s^{t\wedge\sigma} \beta_\tau(u)\rd \tau}
\leq e^{C(t-s)} e^{\left(1+\frac{1}{\log N}\right)\log \left(\frac{\Im[z_{s}(u)]}{\Im[z_{t\wedge\sigma}(u)]}\right)}
=e^{C(t-s)} \left(\frac{\Im[z_{s}(u)]}{\Im[z_{t\wedge\sigma}(u)]}\right)^{1+\frac{1}{\log N}}
\leq C  \frac{\Im[z_s(u)]}{\Im[z_{t\wedge\sigma}(u)]}.
\end{align*}
In the last equality, we used the estimate \eqref{e:z0ztbound}  which shows that $\Im[z_{s}(u)]/\Im[z_{t\wedge\sigma}(u)] \leq C N$.
Combining the above inequality with \eqref{eqn:betabd} we can bound the last term in \eqref{e:midgronwall} by
\begin{align}\begin{split}\label{e:term2}
&C\int_0^{t\wedge\sigma}\frac{\Im[ m_{s}(z_s(u))]}{\Im[z_{t\wedge\sigma}(u)]}\left(\frac{sM(\log N)^{2}}{N}
+\frac{C(\log N)^{1+\delta}}{N\Im[z_s(u)]}\right)\rd s\\
\leq& \frac{CM(\log N)^{2}}{N\Im[z_{t\wedge \sigma}(u)]}\int_0^{t\wedge\sigma}s\Im[m_s(z_s(u))]\rd s+\frac{C(\log N)^{2+\delta}}{N\Im[z_{t\wedge \sigma(u)}]} ,
\end{split}\end{align}
where we used \eqref{e:dzt2} and that $\log (\Im[z_0(u)/z_{t\wedge \sigma}(u)])=\OO(\log N)$ from \eqref{e:z0ztbound}.  Since $|V' (z) | \leq C$, it follows from \eqref{def:zt} that $\Im[m_s(z_s(u))]=-\del_s \Im[z_s(u)] +\OO(1)$. Therefore we can bound the integral term in \eqref{e:term2} by,
\begin{align}\begin{split}\label{e:term3}
\int_0^{t\wedge\sigma}s\Im[m_s(z_s(u))]\rd s = &\int_0^{t\wedge\sigma}(-\del_s \Im[z_s(u)] )s \rd s + \OO ( (t \wedge \sigma )^2 ) = \OO ( t \wedge \sigma ) .
\end{split}\end{align}
It follows by combining \eqref{e:midgronwall}, \eqref{e:term2} and \eqref{e:term3} that
\begin{align}
\left|\td m_{t\wedge\sigma}(z_{t\wedge\sigma}(u))-m_{t\wedge\sigma}(z_{t\wedge\sigma}(u))\right|
\leq  C\left(\frac{({t\wedge\sigma})M (\log N)^{2} +(\log N)^{2+\delta}}{N\Im[z_{t\wedge\sigma}(u)]}\right).
\end{align}
%
Therefore on the event $\Omega$,
\beq
\left|\td m_{t\wedge\sigma}(z_{t\wedge\sigma}(u))-m_{t\wedge\sigma}(z_{t\wedge\sigma}(u))\right|=o\left(\frac{M}{N\Im[z_{t\wedge\sigma}(u)]}\right),
\eeq
provided $t\ll T= (\log N)^{-2}$, and $M=(\log N)^{2+2\delta}$. By Proposition \ref{prop:continuity}, for any $w\in \dom_{t\wedge\sigma}$, there exists some $u\in \cal L$ such that $z_{t\wedge\sigma}(u)\in \dom_{t\wedge\sigma}$, and 
\beq
|z_{t\wedge\sigma}(u)-w|=\OO(N^{-3\fc}).
\eeq
Moreover, on the domain $\cal D_{t\wedge \sigma}$, both $\td m_{t\wedge \sigma}$ and $m_{t\wedge\sigma}$ are Lipschitz with constant $N^{2\fc}$. Therefore
\begin{align}\begin{split}
&\left|\td m_{t\wedge\sigma}(w)-m_{t\wedge\sigma}(w)\right|\leq 
\left|\td m_{t\wedge\sigma}(z_{t\wedge\sigma}(u))-m_{t\wedge\sigma}(z_{t\wedge\sigma}(u))\right|\\
+&
\left|\td m_{t\wedge\sigma}(w)-\td m_{t\wedge\sigma}(z_{t\wedge\sigma}(u))\right|+
\left| m_{t\wedge\sigma}(w)-m_{t\wedge\sigma}(z_{t\wedge\sigma}(u))\right|\\
=& o\left(\frac{M}{N\Im[z_{t\wedge\sigma}(u)]}\right)+\OO\left(\frac{|z_{t\wedge\sigma}(u)-w|}{N^{-2\fc}}\right)=o\left(\frac{M}{N\Im[z_{t\wedge\sigma}(u)]}\right). \label{eqn:q1}
\end{split}\end{align}
If $ \sigma < t$ on the event $\Omega$ then by continuity there must be a point $z \in \dom_\omega$ s.t. 
\beq
| \td m_{ \sigma } (z) - m_{\sigma} (z) | = \frac{ M}{ N \Im [z] } .
\eeq
This contradicts \eqref{eqn:q1}, and so we see that on $\Omega$, $\sigma = t$.  This completes the proof of \eqref{e:diffmm}.
\end{proof}

\subsection{Proof of Corollary \ref{c:eigloc}} \label{sec:corproof}
\begin{proof}[Proof of Corollary \ref{c:eigloc}]
The proof follows a similar argument to \cite[Lemma B.1]{MR2639734}. Recall the function $\eta(x)$ from Remark \ref{r:imagpart}.
Let $S(x+\ri y)=\td m_{t}(x+\ri y)-m_{t}(x+\ri y)$. Fix some $E_0\in[-\fb, \fb]$.  Define
\beq
\td \eta := \inf_{ \eta \geq e^{ K t} N^{1-\fc} } \left\{ \eta : \max_{ E_0 \leq x \leq E_0 + \eta } \eta (x) \leq \eta \right\}.
\eeq

For later use we define
\beq
\td E := \argmax_{E_0 \leq x \leq E_0 +  \td \eta } \eta (x),
\eeq
so that 
\beq
\eta ( \td E ) = \td \eta.
\eeq
We define a test function $f:\bR\rightarrow \bR$, such that $f(x)=1$ on $x\in [-2\fb, E_0]$, and so that $f(x)$ vanishes outside $[-2\fb-1, E_0+\td \eta]$.  We take $f$ so that $f'(x)=\OO(1)$ and $f''(x)=\OO(1)$ on $[-2\fb-1,-2\fb]$ and $f'(x)=\OO(1/\td\eta)$ and $f''(x)=\OO(1/\td\eta^2)$ on $[E_0,E_0+\td\eta]$. By the Helffer-Sj{\"o}strand formula we have,
\begin{align}\begin{split}\label{e:eigHF}
\left|\int_{-\infty}^{\infty} f(x)( \rd \td \mu_t(x)-\rd\mu_t(x))\right|
\leq&
C\int_{\bC_+} ( |f(x)|+|y||f'(x)|)|\chi'(y)||S(x+\ri y)|\rd x\rd y\\
+&C\left|\int_{\bC_+} y\chi(y)f''(x)\Im[S(x+\ri y)]\rd x\rd y\right|.
\end{split}\end{align} 
On the event such that \eqref{e:diffmm} holds, the first term is easily bounded by
\beq\label{e:fterm}
\int_{\bC_+} ( |f(x)|+|y||f'(x)|)|\chi'(y)||S(x+\ri y)|\rd x\rd y
\leq \frac{CM}{N}.
\eeq 
For the second term, recall that $f''(x)=0$ unless $x\in [E_0,E_0+\td\eta]\cup [-2\fb-1,-2\fb]$.  By Remark \ref{r:imagpart} we have the estimate
\beq \label{eqn:ff1}
| \Im [ S ( x + \ri y ) ] | \leq \frac{ C M \log (N) }{ N y }, \qquad y \geq \frac{ e^{ K t }}{ N^{\fc}} =: \eta_\fc.
\eeq
Hence,
\begin{align}
\left|\int_{-2\fb-1\leq x\leq -2\fb} y\chi(y)f''(x)\Im[S(x+\ri y)]\rd x\rd y\right| &\leq \left|\int_{-2\fb-1\leq x\leq -2\fb, |y| \geq \eta_\fc} y\chi(y)f''(x)\Im[S(x+\ri y)]\rd x\rd y\right|  \notag\\
&+ \left|\int_{-2\fb-1\leq x\leq -2\fb, |y| \leq \eta_\fc} | f'' (x) | \rd x\rd y\right|    \notag\\
&\leq   \frac{CM\log N}{N} + \frac{C}{N^{ \fc} } \leq \frac{CM\log N}{N} .
\end{align} 
In the first integral we used the estimate \eqref{eqn:ff1} and in the second we used $|y \Im [ S (x + \i y ) ] | \leq 2$.  For the region $x \in [E_0, E_0 + \td \eta ]$ we do a similar decomposition.  First we bound the region $|y| \leq \td \eta$.  We have,
\begin{align}
\left|\int_{E_0\leq x\leq E_0+\td\eta}\int_{y\leq \td \eta} y\chi(y)f''(x)\Im[S(x+\ri y)]\rd x\rd y\right| &\leq \left|\int_{E_0\leq x\leq E_0+\td\eta}\int_{y\leq \eta_\fc} y\chi(y)f''(x)\Im[S(x+\ri y)]\rd x\rd y\right|   \notag\\ 
&+ \left|\int_{E_0\leq x\leq E_0+\td\eta}\int_{ \eta_\fc \leq y\leq \td \eta} y\chi(y)f''(x)\Im[S(x+\ri y)]\rd x\rd y\right| \notag\\
&\leq \frac{ C \eta_\fc}{ \td \eta} + \frac{ C \log(N) M}{N} \leq \frac{ C \log(N) M}{N}.
\end{align}
For the first integral we used $y|  \Im [S (x + \i y ) ]  |\leq 2$ and in the second region we used \eqref{eqn:ff1}.   For the other region we integrate by parts, 
\begin{align}\begin{split}\label{exp3}
&\int_{E_0\leq x\leq E_0+\td\eta}\int_{y\geq \td \eta} y\chi(y)f''(x)\Im[S(x+\ri y)]\rd x\rd y 
=
-\int_{E_0\leq x\leq E_0+\td \eta} f'(x)\td \eta \Re[S(x+\ri \td \eta)]\rd x\\
-&\int_{E_0\leq x\leq E_0+\td \eta}\int_{y\geq  \td \eta} f'(x)\del_y(y\chi(y))\Re[S(x+\ri y)]\rd x\rd y.
\end{split}\end{align}
By the definition of $\td \eta$, the estimate \eqref{e:diffmm} holds in the region $x+ \i y$ for $x \in [E_0 , E_0 + \i \td \eta]$ and $y \geq \td \eta$.  Hence, both terms are easily estimated by $C  \log(N) M / N$.
The above estimates imply
\begin{align}
\left|\int_{-\infty}^{\infty} f(x)( \rd \td \mu_t(x)-\rd\mu_t(x))\right|
\leq\frac{CM\log N}{N}.
\end{align}
We can now prove the lower bound of \eqref{e:eigloc}.  We have,
\begin{align}\begin{split}
&|\{i:\lambda_i(t)\leq E_0\}|
\leq N\int_{-\infty}^{\infty}f(x)\rd \td\mu_t(x)
\leq N\int_{-\infty}^{E_0+\td\eta}\rd \mu_t(x)+CM\log N.
\end{split}\end{align}
If $\td \eta = e^{ K t } / N^{\fc}$ then the lower bound of \eqref{e:eigloc} follows by taking 
$E_0=\gamma_{i-CM\log N}(t)-  N^{-\fc+1}$. If $\td \eta =\eta(\td E)> e^{ K t }  / N^{\fc}$, then by the defining relation of the function $\eta(E)$ as in \eqref{defetax}, we have $\tilde \eta\Im[m_t(\tilde E+\ri \td \eta)]=e^{Kt}M\log N/N$. We calculate
\beq
\int_{E}^{E+\td\eta}\rd \mu_t(x)\leq \int_{E}^{E+\td\eta}\frac{2\td\eta^2}{(x-\td E)^2+\td \eta^2}\rd \mu_t(x)\leq2\td\eta\Im[m_t(\td E+\ri \td\eta)]=\frac{2e^{Kt}M\log N}{N}.
\eeq
Hence,
\beq
|\{i:\lambda_i(t)\leq E_0\}| \leq N \int_{- \infty}^{ E_0} \d \mu_t (x) + C M \log (N).
\eeq
The lower bound then follows by taking $E_0=\gamma_{i-CM\log N}(t)$. 
  The upper bound of \eqref{e:eigloc} is proven similarly.
\end{proof}

%
%

\section{Mesoscopic Central Limit Theorem}\label{s:mesoCLT}

In this section we prove a mesoscopic central limit theorem for $\beta$-DBM \eqref{DBM}. We recall the parameters $\delta$ and $M$ defined at the beginning of Section \ref{s:rigidity}. In this section, we fix scale parameters, $\eta_*$ and $r$ such that $N^{-1}\leq \eta_*\ll r\leq 1$.  If we assume that the initial data $\bm\la(0)$ is regular down to the scale $\eta_*$, on the interval $[E_0-r, E_0+r]$, we can prove that after time $t\gg \eta_*$, the linear statistics satisfy a central limit theorem on the scale $\eta\ll t$.  The precise definition of regularity is the following assumption.
\begin{assumption}\label{a:asumpL}
We assume that the initial data satisfies the following two conditions.
\begin{enumerate}
\item There exists some finite constant $\fa$, such that $-\fa\leq \la_1(0)\leq \la_2(0)\cdots\leq \la_N(0)\leq \fa$;
\item There exists some finite constant $\fd$, such that 
\beq\label{e:asumpm0}
\fd^{-1}\leq \Im[m_0(z)]\leq \fd,
\eeq
uniformly for any $z\in\{E+\ri\eta: E\in[E_0-r, E_0+r], \eta_*\leq \eta\leq 1\}$. 
\end{enumerate}
\end{assumption}
Under the above assumption we can prove the following mesoscopic central limit theorem for the Stieltjes transform.
\begin{theorem}\label{t:mesoCLT}
Suppose $V$ satisfies Assumption \ref{a:asumpV}, and moreover that $V$ is  $C^5$.   Fix small constant $\delta>0$, $M=(\log N)^{2+2\delta}$, and $N^{-1}\leq \eta_*\ll r\leq 1$, and assume that the  initial data $\bm\la (0)$ satisfies Assumption \ref{a:asumpL}. 
For any time $t$ with $\eta_*\ll t\ll (\log N)^{-1}r\wedge (\log N)^{-2}$, the normalized Stieltjes transform $\Gamma_t(z)\deq N\Im[z]\left(\td m_t(z)-m_t(z)\right)$ is asymptically a Gaussian field on $\{E+\ri \eta: E\in[E_0-r/2, E_0+r/2], M^2/N\ll \eta\ll t/(M\log N)\}$.
We have for any $z_1,z_2,\cdots, z_k\in \{E+\ri \eta: E\in[E_0-r/2, E_0+r/2], M^2/N\ll \eta\ll t/(M\log N)\}$, the joint characteristic function of $\Gamma_t(z_1), \Gamma_t(z_2),\cdots, \Gamma_t(z_k)$ is given by
\begin{align}\begin{split}\label{e:variance}
\bE\left[\exp\left\{\ri\sum_{j=1}^k a_j \Re[\Gamma_t(z_j)]+b_j \Im[\Gamma_t(z_j)]\right\}\right]
=&\exp\left\{\sum_{1\leq j,\ell\leq k}\Re\left[\frac{(a_j-\ri b_j)(a_\ell+\ri b_\ell)\Im[z_j]\Im[z_\ell]}{2\beta(z_j-\bar{z}_\ell)^2}\right]\right\}\\
+&\OO\left(\frac{M^2}{N\min_j\{\Im[z_j]\}}+\frac{M\log N\max_j\{\Im[z_j]\}}{t}\right).
\end{split}\end{align}
\end{theorem}
By standard arguments the above theorem implies the following central limit theorem for mesoscopic linear statistics.
\begin{corollary}\label{c:mesoCLT}
Under the assumptions of Theorem \ref{t:mesoCLT}, the following holds for any compactly supported $C^2$ function $\psi$. Let $M^2/N\ll \eta\ll t$, $E\in[E_0-r, E_0+r]$, and define
\begin{align*}
\psi_{\eta,E}(x)=\psi\left(\frac{x-E}{\eta}\right).
\end{align*}
The normalized linear statistics converges to a Gaussian
\beq
\cal L(\psi_{\eta,E})\deq \sum_{i=1}^N \psi_{\eta,E}(\lambda_i(t))-N\int_{\bR} \psi_{\eta,E}(x) \rd \mu_t(x)\rightarrow N(0, \sigma_\psi^2),
\eeq
in distribution as $N\rightarrow \infty$, where 
\beq
\sigma_\psi^2\deq \frac{1}{2\beta\pi^2}\int_{\bR^2} \left(\frac{\psi(x)-\psi(y)}{x-y}\right)^2\rd x\rd y.
\eeq
\end{corollary}

\subsection{Regularity of the  Stieltjes transform of the limit measure-valued process}
In this subsection we analyze the differential equation of the Stieltjes transform of the limit measure-valued process \eqref{e:mtzt} under the assumptions of Theorem 
\ref{t:mesoCLT}.  We will need some regularity results for $m_t$.  First we prove some preliminary estimates.  The following two estimates are standard.

\begin{lemma}
Under the assumptions of Theorem \ref{t:mesoCLT}, we have, for any interval $I=[E-\eta, E+\eta]$ with $E\in[E_0-r, E_0+r]$ and $\eta\in [4\fd^2\eta_*, 1]$, the estimate
\beq \label{e:bounddensity}
\frac{|I|N}{16\fd^3}\leq |\{i:\lambda_i(0)\in I\}|\leq \fd |I|N .
\eeq
\end{lemma}
\begin{proof}
For the upper bound, by taking $z=E+\ri \eta$, we have
\beq
\fd \geq \Im[m_0(E+\ri\eta)]\geq \frac{1}{N}\sum_{i:\lambda_i\in I}\frac{\eta}{(\lambda_i-E)^2+\eta^2}\geq \frac{|\{i:\lambda_i\in I\}|}{2N\eta }.
\eeq
For the lower bound, let $\eta_1=\eta/(4\fd^2)\geq \eta_*$, we have
\begin{align}\begin{split}
\fd^{-1}\leq &\Im [m_0(E+\ri\eta_1)]=\frac{1}{N}\sum_{i:\lambda_i\in I}\frac{\eta_1}{(\lambda_i-E)^2+\eta_1^2}+\frac{1}{N}\sum_{i:\lambda_i\not\in I}\frac{\eta_1}{(\lambda_i-E)^2+\eta_1^2}\\
\leq& \frac{|\{i:\lambda_i\in I\}|}{N\eta_1}+\frac{\eta_1}{\eta}\frac{1}{N}\sum_{i=1}^N\frac{2\eta}{(\lambda_i-E)^2+\eta^2}
\leq \frac{|\{i:\lambda_i\in I\}|}{N\eta_1}+\frac{2\eta_1}{\eta}\Im[m_0(E+\ri\eta)]\\
\leq &\frac{4\fd^2|\{i:\lambda_i\in I\}|}{N\eta}+\frac{1}{2\fd},
\end{split}
\end{align}
and the lower bound follows by rearranging.
\end{proof}

\begin{corollary}\label{c:mszsbound}
Assume the conditions  of Theorem \ref{t:mesoCLT}.  Let  $u=E+\ri\eta$ with $E\in[E_0-r, E_0+r]$ and $\eta\in [\eta_*, 1]$.  There exists a constant $C>0$ so that if $\Im[z_t(u)]> 0$, then 
\beq \label{e:demtbound}
|m_t(z_t(u))|\leq C\log N,
\eeq
and 
\beq \label{e:derbound}
|\del_t m_t(z_t(u))|\leq C\log N.
\eeq
\end{corollary}
\begin{proof}
For $t=0$, let $\eta_1=4\fd^2\eta$. By a dyadic decomposition we have
\begin{align}\begin{split}
|m_0(u)|
\leq&\frac{1}{N}\left( \sum_{|\lambda_{i}-E|\leq \eta_1}\frac{1}{\eta}+\sum_{k=1}^{\lfloor -\log_2(\eta_1)\rfloor}\sum_{2^{k}\eta_1\geq |\lambda_{i}-E|\geq 2^{k-1}\eta_1}\frac{1}{|\lambda_i-E|}+\sum_{|\lambda_{i}-E|\geq 1/2}\frac{1}{|\lambda_i-E|}\right)\\
\leq& 2\fd\eta_1/\eta-4\fd\log_2 \eta_1+2\leq C\log N.
\end{split}\end{align}
By Proposition \ref{e:gbound} we have that $|\del_z V' (z) |\leq C$ and $| g(z, x) | \leq C$ and so
\beq \label{tt1}
\left|\del_s m_s(z_s(u)\right|\leq C( |m_s(z_s)|+1),
\eeq
and therefore,
\beq \label{tt2}
|m_t(z_t(u))|\leq e^{Ct}(|m_0(z_0(u))|+1)=\OO(\log N).
\eeq
The claim follows. 
\end{proof}
We now derive estimates on quantities appearing in our analysis of $m_t$.
\begin{lemma}\label{l:zsinDs}
Assume that  Assumption \ref{a:asumpL} holds. Let $t\ll (\log N)^{-1}r\wedge (\log N)^{-2}$, and $u\in\{E+\ri\eta: E\in[E_0-r, E_0+r],\eta\in [\eta_*, 1]\}$. If $z_t(u)\in \{E+\ri\eta: E\in[E_0-r/2, E_0+r/2], M^2/N\ll \eta\ll t\}$, then for $0\leq s\leq t$, we have that $z_s\in \cal D_s$ as defined in \eqref{def:dom}, and moreover,
\begin{align}\label{e:integralbound}
\int_0^t\frac{\rd s}{\Im[z_s]^p}\leq 2\fd \int_0^t\frac{\Im[m_s(z_s)]}{\Im[z_s]^p}\rd s \leq \begin{cases} C\log N, & 
p=1 \\ \frac{C}{\Im[z_t]^{p-1}},& p>1 \end{cases}
\end{align}
\end{lemma} 
\begin{proof}
Let $u$ be as in the statement of the lemma and denote $z_s = z_s (u)$.  By \eqref{e:demtbound}, we have $|\Re[z_s]|\leq r+Ct\log N\leq 3\fb-s$ and $\Im[z_s]\leq 1+Ct\log N\leq 3\fb-s$, since $t\leq (\log N)^{-2}$.  By the assumption $\fd^{-1}\leq \Im[m_0(u)]\leq \fd$ and the estimate \eqref{e:mtbound}, we have uniformly for any $0\leq s\leq t$, 
\beq\label{e:Immsbound}
(2\fd)^{-1}\leq \Im[m_s(z_s(u))]\leq 2\fd,
\eeq
since $t\ll1$. Moreover, by \eqref{e:msmtbound}, we have $\Im[z_s]\geq c\Im[z_t]\gg M^2/N$. Therefore, 
\begin{align*}
\frac{e^{Ks}M\log N}{N\Im[m_s(z_s)]}\vee \frac{e^{Ks}}{N^{\fc}}\ll \frac{M^2}{N}\ll \Im[z_s].
\end{align*}
It follows that $z_s(u)\in \dom_s$. 
Since $\Im[m_s(z_s)]\geq (2\fd)^{-1}$, we have
\begin{align}
\int_0^t\frac{\rd s}{\Im[z_s]^p}\leq 2\fd \int_0^t\frac{\Im[m_s(z_s)]}{\Im[z_s]^p}\rd s. 
\end{align}
The case $p=1$ estimate of
\eqref{e:integralbound} follows from \eqref{e:propzt} by using the estimate  $\Im[u]/\Im[z_t(u)]\leq CN$ of \eqref{e:z0ztbound}.  The case $p>1$ follows from \eqref{e:propzt}.
\end{proof}

\begin{lemma}
The following holds under the assumptions of Theorem \ref{t:mesoCLT}.  Let  $u=E+\ri\eta$ with $E\in[E_0-r, E_0+r]$ and $\eta\in [\eta_*, 1]$.   There exists a uniform constant $c>0$ so that If $\Im[z_t(u)]>0$, then 
\beq\label{e:boundRem}
1-t\Re[\del_z m_0(u)]\geq c
\eeq
\end{lemma}
\begin{proof}
By the upper bound in \eqref{e:ztbound}, since $\Im[z_t(u)]\geq 0$, we have
\beq \label{e:boundeta}
\eta=\Im[u]\geq \frac{1-e^{-Ct}}{C}\Im[m_0(u)]\geq \left(t-\frac{Ct^2}{2}\right)\Im[m_0(u)].
\eeq
We write the LHS of \eqref{e:boundRem} as
\beq \label{e:boundRem2}
1-t\Re[\del_z m_0(u)]=1-\frac{t}{\eta}\Im[m_0(u)]+\frac{t}{N}\sum_{i=1}^{N}\frac{2\eta^2}{|\lambda_i(0)-u|^4}.
\eeq
We consider the following  two cases:
\begin{enumerate}
\item If $\eta\geq 2\fd t$, then by \eqref{e:boundRem2} and assumption \eqref{e:asumpm0}, $1-t\Re[\del_z m_0(u)]\geq 1/2$.
\item If $\eta<2\fd t$, let $\eta_1=\eta\vee 4\fd^2 \eta_*\leq 4\fd^2\eta$. By combining \eqref{e:boundeta}, \eqref{e:boundRem2} and \eqref{e:bounddensity}, we have
\begin{align}\begin{split}
1-t\Re[\del_z m_0(u)]
\geq& -\frac{Ct^2}{2\eta}\Im[m_0(u)]+\frac{t}{N}\sum_{i:|\la_i(0)-E|\leq \eta_1}\frac{2\eta^2}{(2\eta_1^2)^2}\\
\geq& -\frac{C\fd t^2}{2\eta}+\frac{t}{2^{10}\fd^9\eta}
= \frac{t}{2^{10}\fd^9\eta}\left(1-2^9C\fd^{10} t\right)\geq \frac{1}{2^{12}\fd^{10}},
\end{split}\end{align}
where we used $t\ll1$.
\end{enumerate}
\end{proof}

In the following we derive the regularity of the Stieltjes transform of the limiting measure-valued process \eqref{e:mtzt}. As a preliminary we study the flow map $u\rightarrow z_s(u)$, and prove that it is Lipschitz. 
\begin{proposition}\label{prop:dzms}
Under the assumptions of Theorem \ref{t:mesoCLT} we have the following.  Let $u=E+\ri\eta$, such that $E\in[E_0-r, E_0+r]$ and $\eta\in [\eta_*, 1]$. If $\Im[z_t(u)]> 0$, then for $0\leq s\leq t$,
\begin{align}
\label{e:duzsbound}& c \leq |\del_x z_s(u)|, |\del_y z_s(u)| \leq C, \\
\label{e:dzmsbound}&|\del_z m_s(z_s(u))|=\OO\left(t^{-1}\right).
\end{align}
where the constants depend on $V'$ and $\fd$.
\end{proposition}

\begin{proof}
For $s=0$, by \eqref{e:boundeta} we have
\begin{align}\begin{split}\label{e:initialbound}
|\del_{z}m_0(u)|&\leq \frac{1}{N}\sum_{i=1}^N\frac{1}{|\la_i(0)-u|^2}=\frac{\Im[m_0(u)]}{\Im[u]}=\OO\left( t^{-1}\right).
\end{split}\end{align}
By taking derivative with respect to $x$ on both sides of \eqref{def:zt}, we get
\beq\label{duzt}
\del_s\del_x z_s(u)=-\del_z m_s(z_s(u))\del_x z_s(u)+\frac{\del_x V'(z_s(u))}{2},\quad \del_x z_0(u)=1,
\eeq 
where $\del_x V'(z_s(u))=\del_z V'(z_s(u))\del_x z_s(u)+\del_{\bar z}V'(z_s(u))\del_x \bar{z}_s(u)$.
By taking derivative with respect to $x$ on both sides of $\eqref{e:mtzt}$, we have
\begin{align}\begin{split}\label{dumt}
&\del_s \left( \del_z m_s(z_s(u))  \right) \del_x z_s(u) +\del_z m_s(z_s(u))\del_s\del_x z_s(u)\\
=&\frac{\del_zm_s(z_s(u))\del_z V'(z_s(u))\del_x z_s(u) +m_s(z_s(u))\del_x \del_z V'(z_s(u))}{2} +\int_{\bR} \del_x g(z_s(u), w)\rd \mu_s(w),
\end{split}\end{align}
where $\del_x g(z_s(u), w)=\del_z g(z_s(u),w)\del_xz_s(u)+\del_{\bar z} g(z_s(u),w)\del_x\bar{z}_s(u)$.
Note that $\del_x z_0(u)=\del_x(x+\ri y)=1$.   We define
\begin{align}
\sigma=t\wedge \inf_{s\geq0} \{\del_x z_s(u)=0\}.
\end{align}
Then $0 < \sigma \leq t \ll(\log N)^{-2}$, and for any $0\leq s<\sigma$ we have $|\del_x \bar z_s(u)| = |\del_x z_s(u)|$.

By combining \eqref{duzt} and \eqref{dumt}, and rearranging we have
\beq \label{derdtmt}
\del_s \left[ \del_z m_s(z_s(u)) \right]=(\del_z m_s(z_s(u)))^2+2\del_z m_s(z_s(u)) b_s+c_s,
\eeq
where 
\begin{align}\begin{split}\label{def:bscs}
b_s = \frac{\del_x V'(z_s(u))}{4\del_x z_s(u)}+\frac{\del_z V'(z_s(u))}{4},\quad c_s=\frac{m_s(z_s(u))\del_x \del_z V'(z_s(u))}{2\del_x z_s(u)}+\frac{\int_{\bR} \del_x g(z_s(u), w)\rd \mu_s(w)}{\del_x z_s(u)}.
\end{split}\end{align}
Under the assumptions of Theorem \ref{t:mesoCLT} we have $\|V' (z) \|_{C^2} \leq C$ and $| \del_z g (z, w) |+|\del_{\bar z}g(z,w)| \leq C$ by Proposition \ref{e:gbound}.   Combining this with Corollary \ref{c:mszsbound} we have $|b_s| + |c_s| \leq C$ for $0 \leq s \leq \sigma$.

First we derive an upper bound for the real part of $\del_z m_s(z_s(u))$. It follows from taking real part on both sides of \eqref{derdtmt} that
\begin{align}\begin{split}
&\del_s \Re[\del_z m_s(z_s(u))] \notag\\
=& ( \Re [ \del_z m_s (z_s (u) ])^2 - ( \Im [ \del_z m_s (z_s (u) )  ] )^2 + 2 \Re [ \del_z m_s (z_s (u) )  ] \Re [ b_s ] - 2 \Im [ \del_z m_s (z_s (u) ) ] \Im [ b_s ] + \Re [ c_s ] \notag\\
\leq&  ( \Re [ \del_z m_s (z_s (u) )  ] )^2 + 2 \Re [ \del_z m_s (z_s (u) )  ] \Re [ b_s] + \Im [ b_s ] ^2 + \Re [ c_s ] \notag\\
= & ( \Re [ \del_z m_s (z_s (u) )  + b_s ] )^2 + \Re [ c_s - b_s^2 ].
\end{split}\end{align}
Therefore, we derive
\begin{align}\begin{split}
\del_s  ( \Re[\del_z m_s(z_s(u))] )_+ 
\leq& \left( ( \Re[\del_z m_s(z_s(u))] )_+ +C]\right)^2 +C\log N,
\end{split}\end{align}
with initial data $( \Re[\del_z m_0(z_0(u))] )_+\leq(1-c)/{t}$ from \eqref{e:boundRem}.
The above ODE is separable and by solving it explicitly and using the fact that 
 $\sqrt{\log N}t\ll 1$, we get 
\begin{align}\begin{split}
\Re[\del_z m_s(z_s(u))]
\leq& \sqrt{C\log N}\tan \left(\arctan\left(\frac{(1-c)/t+C}{\sqrt{C\log N}}\right)+\sqrt{C\log N}s\right)\\ 
\asymp&\sqrt{C\log N}\tan \left(\frac{\pi}{2}-\frac{(c-Ct)\sqrt{C\log N }t}{1-c+Ct}\right)\asymp\frac{1-c}{c t},
\end{split}\end{align}
uniformly for $0\leq s\leq \sigma$.
Therefore, there exists some constant $C$, so that $\Re[ \del_z m_s (z_s (u) ) ]\leq C/t$, uniformly for any $0\leq s\leq \sigma$. 

Using this we derive from \eqref{derdtmt} that,
\begin{align}\begin{split}
\del_s|\del_z m_s(z_s(u))|^2
=&2\Re[\del_s \del_z m_s(z_s(u)) \del_z \bar m_s(z_s(u))]\\
=& 2\Re[\del_z m_s(z_s(u))]| \del_z m_s(z_s(u))|^2
+4\Re[b_s] |\del_z m_s(z_s(u))|^2+2\Re[c_s \del_z \bar m_s(z_s(u))]\\
\leq& \frac{C}{t}| \del_z m_s(z_s(u))|^2+C(t\log N)^2.
\end{split}\end{align}
It follows by Gronwall's inequality that $|\del_z m_s(z_s(u))|=\OO(1/t)$ uniformly for $0\leq s\leq \sigma$.  Notice that $\del_x z_0(u)=1$, and that \eqref{duzt} implies
\begin{align}
\del_xz_s(u)=e^{\int_0^s -\del_z m_s(z_s(u)) + \frac{\del_z V'(z_s(u))}{2}+\frac{\del_{\bar z}V'(z_s(u))\del_x\bar{z}_s}{2\del_x z_s(u)}\rd \tau} \asymp 1.
\end{align}
uniformly for $0\leq s\leq \sigma$. Therefore, $\sigma=t$ and the estimates  \eqref{e:dzmsbound} and $|\del_xz_t(u)| \asymp 1$ are immediate consequences. The estimate $|\del_yz_t(u)| \asymp 1$ follows from the same argument.
%
%
\end{proof}

Finally, we have the following results for the regularity of $m_t (w)$.
\begin{corollary}\label{c:uniformest}
Suppose that the assumptions of Theorem \ref{t:mesoCLT} hold, and let $\eta_*\ll t\ll (\log N)^{-1}r\wedge (\log N)^{-2}$.  We have,
\begin{enumerate}
\item[\rn{1})] For any $w\in \{E+\ri\eta: E\in[E_0-3r/4, E_0+3r/4], 0< \eta\leq 3/4\}$, we have that $z_t^{-1}(w)\subset \{E+\ri\eta: E\in[E_0-r, E_0+r], \eta_*\leq \eta\leq 1\}$, and $\del_zm_t(w)=\OO(1/t)$.
\item[\rn{2})] Fix $u\in\{E+\ri\eta: E\in[E_0-r, E_0+r],\eta\in [\eta_*, 1]\}$. If $z_t(u)\in \{E+\ri\eta: E\in[E_0-r/2, E_0+r/2], 0< \eta\ll t\}$, then for $0\leq s\leq t$, and any $w\in \bC_+$ such that $|w-z_s(u)|\leq \Im[z_s(u)]/2$, we have $|\del_z m_s(w)|=\OO(1/t)$.
\end{enumerate}
In both statements, the implicit constants depend on $V$ and $\fd$.
\end{corollary}
\begin{proof}
We first consider the first statement in {\rn{1}}). Uniformly for any $u\in\{E+\ri\eta: E\in[E_0-r, E_0+r], \eta_*< \eta\leq 1\}\cap \Omega_t$ (with $\Omega_t$ as in Proposition \ref{prop:invflow}), we have by \eqref{e:ztbound}, \eqref{e:demtbound} and \eqref{def:zt}, that there exists a constant $C$ depending on $V$ and $\fd$, such that 
\begin{align}\begin{split}
\max\left\{0, \Im[u]-2tC\Im[m_0(u)]\right\}&\leq \Im[z_t(u)]\leq e^{Ct}\left(\Im[u]-\frac{1-e^{-Ct}}{C}\Im[m_0(u)]\right),\\
\Re[u]-Ct\log N &\leq \Re[z_t(u)]\leq \Re[u]+Ct\log N.
\end{split}\end{align}
By Proposition \ref{prop:invflow}, $z_t$ is surjective from $\Omega_t$ onto $\bC_+$.
The first statement in \rn{1}) follows from the assumptions $t\gg\eta_*$ and $r\gg t\log N$. 
The second statement  in \rn{1}),  is then a consequence of \eqref{e:dzmsbound} and the equality 
$\del_zm_t(w)=\del_zm_t(z_t(z_t^{-1}(w)))$.

For \rn{2}), since $\Im[m_0(u)]=\OO(1)$, it follows from \eqref{e:ztbound} that $ \Im[u]=t\Im[m_0(u)]+o(t)$. If $s\leq t/2$, then we see that by \eqref{e:ztbound} that $t / C \leq \Im [ z_s (u ) ] \leq Ct $ for some $C>0$.  
Furthermore, by \eqref{e:demtbound} and \eqref{def:zt} we see that $\Re[u] - C t \log (N) \leq \Re [ z_s (u) ] \leq \Re[u] + C t \log (N)$.  
We also observe that $\Im[w]\geq t/2C$. It follows from the same argument as in \rn{1}) that $\{w\in \bC_+: |w-z_s(u)|\leq \Im[z_s(u)]/2\} \subseteq z_s(\{E+\ri\eta: E\in[E_0-r, E_0+r],\eta\in [\eta_*, 1]\}\cap \Omega_s)$. Therefore, by \eqref{e:mtbound}, uniformly for $\{w\in \bC_+: |w-z_s(u)|\leq \Im[z_s(u)]/2\}$, $\Im[m_s(w)]=\OO(\Im[m_0(z_s^{-1}(w))])=\OO(1)$, and therefore
\beq
|\del_z m_s(w)|\leq \frac{\Im[m_s(w)]}{\Im[w]}=\OO\left(\frac{1}{t}\right).
\eeq

If $s\geq t/2$, from \rn{1}), uniformly for any $w\in \{E+\ri\eta: E\in[E_0-3r/4, E_0+3r/4], 0< \eta\leq  3/4\}$, $\del_zm_s(w)=\OO(1/s)=\OO(1/t)$. Moreover, we have $\{w\in \bC_+: |w-z_s(u)|\leq \Im[z_s(u)]/2\} \subseteq \{E+\ri\eta: E\in[E_0-3r/4, E_0+3r/4], 0< \eta\leq 3/4\}$. The statement follows. 
%
%
%
%
%
\end{proof}

%


\subsection{Proof of Theorem \ref{t:mesoCLT}}

Using regularity of $m_t$ and the local law we infer the following regularity for the empirical Stieltjes transform $\td m_t$.
\begin{lemma}\label{dpbound}
Suppose that the assumptions of Theorem \ref{t:mesoCLT} hold.  Let $\eta_*\ll t\ll (\log N)^{-1}r\wedge (\log N)^{-2}$. Fix $u\in\{E+\ri\eta: E\in[E_0-r, E_0+r],\eta\in [\eta_*, 1]\}$. If $z_t(u)\in \{E+\ri\eta: E\in[E_0-r/2, E_0+r/2], 0< \eta\ll t\}$, then on the event $\Omega$ (as defined in the proof of Proposition \ref{p:error}), we have the following estimate uniformly for $0\leq s\leq t$,
\beq \label{dptdms}
\del_z^p \td m_s(z_s(u))=\OO\left(\frac{M}{N\Im[z_s(u)]^{p+1}}+\frac{1}{t\Im[z_s(u)]^{p-1}}\right)
\eeq
\end{lemma}


\begin{proof}
The estimate \eqref{dptdms} is a consequence of the following two statements.
\begin{align}
\label{dpmsdiff}&\del_z^p \left(\td m_s(z_s(u))-m_s(z_s(u))\right)=\OO\left(\frac{M}{N\Im[z_s(u)]^{p+1}}\right),\\
\label{dpms}&\del_z^p m_s(z_s(u))=\OO\left(\frac{1}{t\Im[z_s(u)]^{p-1}}\right).
\end{align}

For \eqref{dpmsdiff}, since both $\td m_s$ and $m_s$ are analytic on the upper half plane, by Cauchy's integral formula
\beq
\del_z^p \left(\td m_s(z_s(u))-m_s(z_s(u))\right)
=\frac{p!}{2\pi \ri}\oint_{\cC} \frac{\td m_s(w)-m_s(w)}{(w-z_s(u))^{p+1}}\rd w ,
\eeq
where $\cC$ is a small contour in the upper half plane centering at $z_s(u)$ with radius $\Im[z_s(u)]/2$. On the event $\Omega$, we use  \eqref{e:diffmm} in Theorem \ref{t:rigidity} to bound the integral by 
\beq
\left|\frac{p!}{2\pi \ri}\oint_{\cC} \frac{\td m_s(w)-m_s(w)}{(w-z_s(u))^{p+1}}\rd w\right|
\leq \frac{p!}{2\pi }\oint_{\cC} \frac{|\td m_s(w)-m_s(w)|}{|w-z_s(u)|^{p+1}}\rd w
=\OO\left(\frac{M}{N\Im[z_s(u)]^{p+1}}\right).
\eeq
For \eqref{dpms}, Cauchy's integral formula leads to
\begin{align}\begin{split}
\left|\del_z^p  m_s(z_s(u))\rd s\right|
\leq \frac{(p-1)!}{2\pi}\oint_{\cC} \frac{\left|\del_z m_s(w)\right|}{|w-z_s(u)|^{p}}\rd w=\OO\left(\frac{1}{t\Im[z_s(u)]^{p-1}}\right).
\end{split}\end{align}
where we used \rn{2}) in Corollary \ref{c:uniformest} which states that $\left|\del_z m_s(w)\right|=\OO(1/t)$.
\end{proof}

By \rn{1}) in Corollary \ref{c:uniformest}, $\{E+\ri \eta: E\in[E_0-r/2, E_0+r/2], M^2/N\ll \eta\ll t\}\subseteq z_t(\{E+\ri\eta: E\in[E_0-r, E_0+r], \eta\in [\eta_*, 1]\}\cap \Omega_t)$. In the following, we fix some $u\in\{E+\ri\eta: E\in[E_0-r, E_0+r], \eta\in [\eta_*, 1]\}$, such that $z_t(u)\in\{E+\ri \eta: E\in[E_0-r/2, E_0+r/2], M^2/N\ll \eta\ll t\}$. 
%
By Lemma \ref{l:zsinDs}, $z_t\in \dom_t$, and the local law of  Theorem \ref{t:rigidity} holds.

We integrate both sides of \eqref{e:diffm}, and get the following integral expression for $\td m_t(z_t)$, 
\begin{align}\label{e:intdiffm}\begin{split}
&\td m_t(z_t)-m_t(z_t)=\int_0^t\left(\td m_s(z_s)-m_s(z_s)\right)\del_z \left(\td m_s(z_s)+\frac{V'(z_s)}{2}\right)\rd s\\
+& \frac{1}{\pi}\int_0^t\int_{\bC} \del_{\bar w} \td g(z_s,w) (\td m_s(w)-m_s(w))\rd^2 w\rd s
+\frac{2-\beta}{\beta N^2}\int_0^t\sum_{i=1}^{N}\frac{\rd s}{(\la_i(s)-z_s)^3}\\
-&\sqrt{\frac{2}{\beta N^3}}\int_0^t\sum_{i=1}^N \frac{{\rm d} B_i(s)}{(\la_i(s)-z_s)^2}.
\end{split}
\end{align}
For the  proof of the mesoscopic central limit theorem, we will show that the first three terms on the righthand side of \eqref{e:intdiffm} are negligible, and the Gaussian fluctuation is from the last term, i.e. the integral with respect to Brownian motion. In the following Proposition, we calculate the quadratic variance of the Brownian integrals.

\begin{proposition}\label{p:var}
Suppose that the  assumptions of Theorem \ref{e:asumpm0} hold.  Fix $u,u'\in\{E+\ri\eta: E_0-r\leq E\leq E_0+r,\eta_*\leq \eta\leq  1\}$. Let $z_t\deq z_t(u)$ and $z_t'\deq z_t(u')$.  If 
\beq
z_t,z'_t\in \{E+\ri \eta: E\in[E_0-r/2, E_0+r/2], M^2/N\ll \eta\ll t\},
\eeq 
and $\Im[z_t]\geq \Im[z_t']$, then
\begin{align}
\label{e:var1}&\frac{1}{N^3}\int_0^t\sum_{i=1}^N \frac{{\rm d}s}{(\la_i(s)-z_s)^4}=\OO\left( \frac{M}{N^3\Im[z_t]^3}+\frac{1}{N^2t\Im[z_t]}\right), \\
\label{e:var2}&\frac{1}{N^3}\int_0^t\sum_{i=1}^N \frac{{\rm d}s}{(\la_i(s)-z_s)^2(\la_i(s)-z_s')^2}=\OO\left( \frac{M}{N^3\Im[z_t]^2\Im[z_t']}+\frac{1}{N^2t\Im[z_t]}\right), \\
\label{e:var3}&\frac{1}{ N^3}\int_0^t\sum_{i=1}^N \frac{{\rm d} s}{(\la_i(s)-\bar{z}_s)^2(\la_i(s)-z_s')^2}=-
\frac{1}{N^2(\bar{z}_t-z_t')^2}
+\OO\left( \frac{M}{N^3\Im[z_t]^2\Im[z_t']}+\frac{1}{N^2t\Im[z_t]}\right).
\end{align}
\end{proposition}
\begin{proof}
Since $\Im[m_0(z_0)]=\Im[m_0(u)] \asymp 1$, by \eqref{e:ztbound} and \eqref{e:mtbound}, we have $\Im[z_s] \asymp\Im[z_t]+(t-s)$ and $\Im[z_s'] \asymp \Im[z_t']+(t-s)$. Since $\Im[z_t]\geq \Im[z_t']$, there exists a constant $c$ depending on $V$ and $\fd$, such that uniformly for $0\leq s\leq t$, $\Im[z_s]\geq c\Im[z_s']$.

For \eqref{e:var1}, the lefthand side can be written as the derivative of the Stieltjes transform $\td m_s$ at $z_s$, and so
\begin{align}\begin{split}
\left|\frac{1}{6N^2}\int_0^t\del^3_z\td m_s(z_s)\rd s\right|
\leq& \frac{C}{6N^2}\int_0^t\left(\frac{M}{N\Im[z_s]^4}+\frac{1}{t\Im[z_s]^2}\right)\rd s\\
=&\OO\left( \frac{M}{N^3\Im[z_t]^3}+\frac{1}{N^2t\Im[z_t]}\right),
\end{split}\end{align}
where we used Lemma \ref{dpbound} and \eqref{e:integralbound}.

We write the LHS of \eqref{e:var2}, as a contour integral of $\td m_s$:
\begin{align}\begin{split}\label{e:contourintm}
\frac{1}{N^3}\sum_{i=1}^N \frac{1}{(\la_i(s)-z_s)^2(\la_i(s)-z_s')^2}
=&\frac{1}{2\pi\ri N^2}\oint_{\cal C}\frac{\td m_s(w)}{(w-z_s)^2(w-z_s')^2}\rd w,
\end{split}\end{align}
where if $\Im[z_s]/3\geq |z_s-z'_s|$, then $\cal C$ is a contour centered at $z_s$ with radius $\Im[z_s]/2$. In this case we have $\dist(\cal C, \{z_s, z_s'\})\geq \Im[z_s]/6$.   In the case that $|z_s-z'_s|\geq \Im[z_s]/3$, 
we let $\cal C=\cal C_1\cup \cal C_2$ consist of two contours, where $\cal C_1$ is centered at $z_s$ with radius $\min\{\Im[z_s'],\Im[z_s]\}/6$, and $\cal C_2$ is centered at $z_s'$ with radius $\min\{\Im[z_s'],\Im[z_s]\}/6$. Then in this case we have
$\dist(\cal C_1, z_s')\geq \Im[z_s]/6$ and $\dist(\cal C_2, z_s)\geq \Im[z_s]/6$. In the first case, thanks to Lemma \ref{dpbound} and \rn{2}) in Corollary \ref{c:uniformest}, for $w\in \cal C$ we have
\beq\label{e:tdmsexpand}
\td m_s(w)=\td m_s(z_s)+(w-z_s)\del_z \td m_s(z_s)+(w-z_s)^2\OO\left(\frac{M}{N\Im[z_s]^{3}}+\frac{1}{t\Im[z_s]}\right).
\eeq
Plugging \eqref{e:tdmsexpand} into \eqref{e:contourintm}, we see that the first two terms vanish and
\begin{align}
|\eqref{e:contourintm}|\leq \frac{C}{N^2}\int_{\cal C}\left(\frac{M}{N\Im[z_s]^{5}}+\frac{1}{t\Im[z_s]^3}\right)\rd w
=\OO\left(\frac{M}{N^3\Im[z_s]^{4}}+\frac{1}{N^2t\Im[z_s]^2}\right),\label{e:contourintm1} 
\end{align}
where we used that $|\cal C|\asymp \Im[z_s]$. In the second case, \eqref{e:tdmsexpand} holds on $\cal C_1$. Similarly, for  $w\in \cal C_2$ we have
\beq\label{e:tdmsexpand2}
\td m_s(w)=\td m_s(z_s')+(w-z_s')\del_z \td m_s(z_s')+(w-z_s')^2\OO\left(\frac{M}{N\Im[z_s']^{3}}+\frac{1}{t\Im[z_s']}\right).
\eeq
It follows by plugging \eqref{e:tdmsexpand} and \eqref{e:tdmsexpand2} into \eqref{e:contourintm}, that we can bound \eqref{e:contourintm} by
\begin{align}\begin{split}\label{e:contourintm2}
& \frac{C}{N^2}\left(\int_{\cal C_1}\left(\frac{M}{N\Im[z_s]^{5}}+\frac{1}{t\Im[z_s]^3}\right)\rd w+\int_{\cal C_2}\left(\frac{M}{N\Im[z_s]^{2}\Im[z_s']^{3}}+\frac{1}{t\Im[z_s]^2\Im[z_s']}\right)\rd w\right)\\
=&\OO\left(\frac{M}{N^3\Im[z_s]^{2}\Im[z_s']^2}+\frac{1}{N^2t\Im[z_s]^2}\right),
\end{split}\end{align}
where we used $\Im[z_s]\geq c\Im[z_s']$ and $|\cal C_1|, |\cal C_2|=\OO(\Im[z_s'])$.  The estimate of the LHS of
\eqref{e:var2} follows by combining \eqref{e:contourintm1} and \eqref{e:contourintm2},
\begin{align}\begin{split}
|\eqref{e:var2}|\leq &\frac{C}{N^2}\int_0^t \frac{M}{N^3\Im[z_s]^{2}\Im[z_s']^2}+\frac{1}{N^2t\Im[z_s]^2}\rd s\\
= &\OO\left(\frac{M}{N^3\Im[z_t]^2}\int_0^t \frac{\rd s}{\Im[z_s']^2}+\frac{1}{N^2t\Im[z_t]}\right)\\
=&\OO\left(\frac{M}{N^3\Im[z_t]^2\Im[z_t']}+\frac{1}{N^2t\Im[z_t]}\right),
\end{split}\end{align}
where we used that $\Im[z_s]\geq c\Im[z_t]$ in the second line, and \eqref{e:integralbound} for the last line.
%

Finally, for \eqref{e:var3},
\beq\label{e:mainterm}
\frac{1}{N}\sum_{i=1}^N \frac{1}{(\la_i(s)-\bar z_s)^2(\la_i(s)-z_s')^2}
=\frac{2(\overline{-\td m_s(z_s)}+\td m_s(z_s'))}{(\bar z_s-z_s')^3}+\frac{\overline{\del_{z} \td m_s( z_s)}+\del_z \td m_s(z_s')}{(\bar z_s-z_s')^2}.
\eeq
Note that $|\bar z_s-z_s'|\geq \Im[z_s]+\Im[z_s']\asymp \Im[z_s]$.   For the second term in \eqref{e:mainterm}, we have by \eqref{dptdms},
\begin{align}\begin{split}
\left|\frac{1}{N^2}\int_0^t\frac{\overline{\del_{z} \td m_s( z_s)}+\del_z \td m_s(z_s')}{(\bar z_s-z_s')^2}\right|
\leq& \frac{C}{N^2}\int_0^t\frac{1}{\Im[z_s]^2} \left(\frac{M}{N\Im[z_s']^2}+\frac{1}{t}\right)\rd s\\
=&\OO\left(\frac{M}{N^3\Im[z_t]^2\Im[z_t']}+\frac{1}{N^2t\Im[z_t]}\right).
\end{split}\end{align}
For the first term in \eqref{e:mainterm}, we recall the definition of the vector flow $z_s(u)$ as in \eqref{def:zt}.  Since $\|V'(z)\|_{C^1}=\OO(1)$, we have
\beq
\overline{-\td m_s(z_s)}+\td m_s(z_s')=\del_s (\bar z_s-z_s') +\OO(|\bar z_s-z_s'|).
\eeq
Therefore,
\begin{align}
\begin{split}
\frac{2}{N^2}\int_{0}^{t}\frac{(\overline{-\td m_s(z_s)}+\td m_s(z_s'))}{(\bar z_s-z_s')^3}\rd s
=&
\frac{2}{N^2}\int_{0}^{t}\frac{\del_s (\bar z_s-z_s') }{(\bar z_s-z_s')^3}\rd s
+\OO\left( \frac{1}{N^2}\int_0^t\frac{\rd s}{\Im[z_s]^2}\right)\\
=&-\frac{1}{N^2(\bar z_t-z_t')^2}+\frac{1}{N^2(\bar u-u')^2}+\OO\left(\frac{1}{N^2\Im[z_t]}\right)\\
=&-\frac{1}{N^2(\bar z_t-z_t')^2}+\OO\left(\frac{1}{N^2t^2}+\frac{1}{N^2\Im[z_t]}\right),
\end{split}
\end{align}
where we used $|\bar u-u'|\geq \Im[u]+\Im[u']\geq ct$. This finishes the proof of Proposition  \ref{p:var}.
\end{proof}

\begin{proof}[Proof of Theorem \ref{t:mesoCLT}]
Let the event $\Omega$ be as above.  Thanks to the estimates Theorem \ref{t:rigidity} and Lemma \ref{dpbound} which hold on $\Omega$, we can bound the first term on the RHS of \eqref{e:intdiffm}  by
\begin{align}\begin{split}
\left|\int_0^t\left(\td m_s(z_s)-m_s(z_s)\right)\del_z \left(\td m_s(z_s)+\frac{V'(z_s)}{2}\right)\rd s\right|
\leq& C\int_0^t\frac{M}{N\Im[z_s]}\left(\frac{M}{N\Im[z_s]^2}+\frac{1}{t}\right)\rd s\\
=& \OO\left(\frac{M^2}{(N\Im[z_t])^2}+\frac{M\log N}{Nt}\right),
\end{split}\end{align}
where we used \eqref{e:integralbound}.



For the second term on the righthand side of \eqref{e:intdiffm}, by Proposition \ref{prop:HFbound} we have on the event $\Omega$
\beq
 \left|\frac{1}{\pi}\int_0^t\int_{\bC} \del_{\bar w} \td g(z_s,w) (\td m_s(w)-m_s(w))\rd^2 w\rd t\right|\leq \frac{CtM(\log N)^{2}}{N}.
\eeq

We can rewrite the third term on the righthand side of \eqref{e:intdiffm} as
\beq
\frac{2-\beta}{\beta N^2}\int_0^t\sum_{i=1}^{N}\frac{\rd s}{(\la_i(s)-z_s)^3}
=\frac{2-\beta}{2\beta N}\int_0^t\del_z^2 \td m_s(z_s)\rd s.
\eeq
Thanks to Lemma \ref{dpbound}, and \eqref{e:integralbound} we have
\begin{align}\begin{split}
\left|\int_0^t\del_z^2 \td m_s(z_s)\rd s\right|
\leq& C\int_0^t\left(\frac{1}{N\Im[z_s]^3}
+ \frac{1}{t\Im[z_s]}\right)\rd s=\OO\left(\frac{1}{N (\Im[z_t])^2}+\frac{\log N}{t}\right).
\end{split}
\end{align}
It follows that 
\beq
\left|\frac{2-\beta}{\beta N^2}\int_0^t\sum_{i=1}^{N}\frac{\rd s}{(\la_i(s)-z_s)^3}\right|
=\OO\left(\frac{1}{(N\Im[z_t])^2}+\frac{\log N}{Nt}\right).
\eeq

%
%
%

By combining the above estimates we see that on the event $\Omega$, we have
\beq\label{e:tmmtdiff}
\td m_t(z_t)-m_t(z_t)=\OO\left(\frac{M^2}{(N\Im[z_t])^2}+\frac{M\log N}{Nt}\right)+\sqrt{\frac{2}{\beta N^3}}\int_0^t\sum_{i=1}^N \frac{{\rm d} B_i(t)}{(\la_i(s)-z_s)^2}.
\eeq

In the following we show that the Brownian integrals are asymptotically jointly Gaussian. We fix some $u_j\in\{E+\ri\eta: E_0-r\leq E\leq E_0+r,\eta_*\leq \eta\leq  1\}$, $j=1,2,\cdots, k$ such that \beq
z_t(u_j)\in \{E+\ri \eta: E\in[E_0-r/2, E_0+r/2], M^2/N\ll \eta\ll t\},\quad j=1,2,\cdots, k.
\eeq 
For $1\leq j\leq k$. Let
\beq
X_j(t)= \Im[z_t(u_j)]\sqrt{\frac{2}{\beta N}}\int_0^t\sum_{i=1}^N \frac{{\rm d} B_i(t)}{(\la_i(s)-z_s(u_j))^2},\quad j=1,2,\cdots, k.
\eeq
We compute their joint characteristic function,
\beq\label{e:cfunc}
\bE\left[\exp\left\{\ri\sum_{j=1}^k a_j\Re[X_j(t)]+b_j\Im[X_j(t)]\right\}\right]
\eeq
Since $\sum_{j=1}^k a_j\Re[X_j(t)]+b_j\Im[X_j(t)]$ is  a martingale, the following is also a martingale
\beq
\exp\left\{\ri \sum_{j=1}^k a_j\Re[X_j(t)]+b_j\Im[X_j(t)]\}+\frac{1}{2}\left\langle \sum_{j=1}^k a_j\Re[X_j(t)]+b_j\Im[X_j(t)]\right\rangle\right\}
\eeq
In particular, its expectation is one. By Proposition \ref{p:var}, on the event $\Omega$ (as defined in the proof of Proposition \ref{p:error}), the quadratic variation is given by 
\begin{align}\begin{split}
&\frac{1}{2}\left\langle \sum_{j=1}^k a_j\Re[X_j(t)]+b_j\Im[X_j(t)]\right\rangle\\
=&-\sum_{1\leq j,\ell\leq k}\Re\left[\frac{(a_j-\ri b_j)(a_\ell+\ri b_\ell)\Im[z_t(u_j)]\Im[z_t(u_\ell)]}{2\beta(z_t(u_j)-\overline{z_t(u_\ell)})^2}\right]\\
+&\OO\left( \frac{M}{N\min_j\{z_t(u_j)\}}+\frac{\max_j\{\Im[z_t(u_j)]\}}{t}\right).
\end{split}\end{align}
Therefore,
\begin{align}\begin{split}
\eqref{e:cfunc}
=&\exp\left\{\sum_{1\leq j,\ell\leq k}\Re\left[\frac{(a_j-\ri b_j)(a_\ell+\ri b_\ell)\Im[z_t(u_j)]\Im[z_t(u_\ell)]}{2\beta(z_t(u_j)-\overline{z_t(u_\ell)})^2}\right]\right\} \\+&\OO\left( \frac{M}{N\min_j\{\Im[z_t(u_j)]\}}+\frac{\max_j\{\Im[z_t(u_j)]\}}{t}\right).
\end{split}\end{align}
Since by \eqref{e:tmmtdiff}, 
\begin{align*}
\Gamma_t(z_t(u_j))=X_j(t)+\OO\left(\frac{M^2}{N\Im[z_t]}+\frac{M\log N\Im[z_t(u_j)]}{t}\right),
\end{align*}
and so \eqref{e:variance} follows. This finishes the proof of Theorem \ref{t:mesoCLT}.
\end{proof}

\begin{proof}[Proof of Corollary \ref{c:mesoCLT}]
The corollary follows from Theorem \ref{t:mesoCLT} and the rigidity estimate \ref{t:rigidity} by the same argument as in \cite[Theorem 1.2]{mesoCLT1}.
\end{proof}

\bibliography{References.bib}{}
\bibliographystyle{plain}

\end{document}